\documentclass[twoside,11pt,reqno]{amsart}
\usepackage{amscd}
\usepackage{amssymb}

\makeatletter

\@addtoreset{equation}{section}

\newtheorem{Proposition}{Proposition}[section]
\newtheorem{Lemma}[Proposition]{Lemma}
\newtheorem{Theorem}[Proposition]{Theorem}
\newtheorem{Corollary}[Proposition]{Corollary}

\theoremstyle{remark}
\newtheorem{Remark}[Proposition]{Remark}

\newbox\squ  
\setbox\squ=\hbox{\vrule width.6pt
   \vbox{\hrule height.6pt width.4em\kern1ex
         \hrule height.6pt}%
   \vrule width.6pt}

\newcommand{\Tab}{\operatorname{Tab}}
\newcommand{\col}{\operatorname{col}}
\newcommand{\cdet}{\operatorname{cdet}}
\newcommand{\row}{\operatorname{row}}
\newcommand{\up}{{\operatorname{up}}}
\newcommand{\lo}{{\operatorname{lo}}}

\newcommand{\g}{\mathfrak{g}}
\newcommand{\h}{\mathfrak{h}}
\newcommand{\gl}{\mathfrak{gl}}
\newcommand{\m}{\mathfrak{m}}
\newcommand{\p}{\mathfrak{p}}
\newcommand{\q}{\mathfrak{q}}
\def\mf{\mathfrak}

\def\coloneqq{:=}
\def\deg{\operatorname{deg}}
\def\gr{\operatorname{gr}}
\def\id{\operatorname{id}}
\def\ev{{\operatorname{ev}}}
\def\C{{\mathbb C}}
\def\Z{{\mathbb Z}}
\def\0{{\bar 0}}
\def\1{{\bar 1}}
\def\eps{{\varepsilon}}
\def\llbracket{[\![}
\def\rrbracket{]\!]}

\newdimen\Hoogte    \Hoogte=10pt    
\newdimen\Breedte   \Breedte=10pt   
\newdimen\Dikte     \Dikte=0.5pt    

\newenvironment{Young}{\begingroup
       \def\vr{\vrule height0.8\Hoogte width\Dikte depth 0.2\Hoogte}
       \def\fbox##1{\vbox{\offinterlineskip
                    \hrule height\Dikte
                    \hbox to \Breedte{\vr\hfill##1\hfill\vr}
                    \hrule height\Dikte}}
       \vbox\bgroup \offinterlineskip \tabskip=-\Dikte \lineskip=-\Dikte
            \halign\bgroup &\fbox{##\unskip}\unskip  \crcr }
       {\egroup\egroup\endgroup}

\def\Diagram#1{\relax\ifmmode\vcenter{\,\begin{Young}#1\end{Young}\,}\else%
              $\vcenter{\,\begin{Young}#1\end{Young}\,}$\fi}

\title[Principal $W$-algebras]{\boldmath Principal $W$-algebras for
$\mathrm{GL}(m|n)$}
\author{Jonathan Brown, Jonathan Brundan and Simon M.~Goodwin}
\address{School of Mathematics,
University of Birmingham,
Birmingham, B15 2TT,
UK}
\email{brownjs@for.mat.bham.ac.uk, goodwin@for.mat.bham.ac.uk}
\address{Department of Mathematics, University of Oregon, Eugene, OR 97403, USA}
\email{brundan@uoregon.edu}

\thanks{First and third authors supported by EPSRC grant EP/G020809/1.}
\thanks{Second author supported in part by NSF grant no. DMS-1161094.}

\begin{document}

\begin{abstract}
We consider the (finite) $W$-algebra $W_{m|n}$ attached to the
principal
nilpotent orbit in the general linear Lie superalgebra
$\mathfrak{gl}_{m|n}(\mathbb{C})$. Our main result gives an explicit
description of $W_{m|n}$ as a certain truncation of a
shifted version of the Yangian $Y(\mathfrak{gl}_{1|1})$.
We also show that $W_{m|n}$ admits a triangular decomposition
and construct its irreducible representations.
\end{abstract}

\maketitle

\section{Introduction}

A {\em (finite) $W$-algebra} is a certain filtered deformation of the Slodowy
slice to a nilpotent orbit in a complex semisimple Lie algebra $\mathfrak{g}$.
Although the terminology is more recent, the
construction has its origins in the classic work of Kostant \cite{K}.
In particular Kostant
showed that the
principal $W$-algebra---the one
associated to the principal nilpotent orbit in $\mathfrak{g}$---is
isomorphic to the center of the universal enveloping algebra
$U(\mathfrak{g})$.
In the last few years there has been some substantial progress
in understanding $W$-algebras for other nilpotent orbits,
thanks to works of Premet, Losev and others; see
\cite{L} for a survey. The story is  most complete (also easiest) for
$\mathfrak{sl}_n(\C)$. In this case
the
$W$-algebras are closely related to {\em shifted Yangians}; see
\cite{BK}.

Analogues of $W$-algebras have also been defined for Lie
superalgebras; see for example the work of De Sole and Kac \cite[$\S$5.2]{DK}
(where they are defined in terms of BRST cohomology)
or the more recent paper of
Zhao \cite{Zhao} (which focusses mainly on the queer Lie superalgebra
$\mathfrak{q}_n(\C)$).
In this article we consider the easiest of all the ``super''
situations: the {\em principal $W$-algebra
$W_{m|n}$
for the general linear Lie superalgebra $\mathfrak{gl}_{m|n}(\C)$}.
Our main result gives an explicit isomorphism between $W_{m|n}$ and a certain
truncation of a shifted subalgebra of the Yangian
$Y(\mathfrak{gl}_{1|1})$; see Theorem \ref{T:main}.
Its proof is very similar to the
proof of the analogous result for
nilpotent
matrices of Jordan type $(m,n)$
in $\mathfrak{gl}_{m+n}(\C)$
from \cite{BK}.

The (super)algebra $W_{m|n}$ turns out to be quite close to being supercommutative.
More precisely, we show that
it admits a triangular decomposition $$
W_{m|n} = W_{m|n}^-
W_{m|n}^0 W_{m|n}^+
$$
in which $W_{m|n}^-$ and $W_{m|n}^+$ are
exterior algebras of dimension $2^{\min(m,n)}$
and $W_{m|n}^0$
is a symmetric algebra of rank $(m+n)$; see Theorem~\ref{triangular}.
This implies that all the irreducible $W_{m|n}$-modules are finite
dimensional; see Theorem~\ref{irrclass}.
We show further that they
all arise as certain tensor products
of irreducible $\mathfrak{gl}_{1|1}(\C)$- and
$\mathfrak{gl}_1(\C)$-modules; see Theorem~\ref{fint}. In particular, all irreducible
$W_{m|n}$-modules are of dimension dividing $2^{\min(m,n)}$.
A closely related assertion is that all irreducible
highest weight representations of $Y(\mathfrak{gl}_{1|1})$
are tensor products of evaluation modules;
this is similar to a well-known phenomenon for $Y(\mathfrak{gl}_2)$
going back to \cite{T}.

Some related results
about $W_{m|n}$ have been obtained
independently by Poletaeva and Serganova \cite{PS}.
In fact, the connection between $W_{m|n}$ and the Yangian
$Y(\mathfrak{gl}_{1|1})$ was foreseen long ago by Briot and Ragoucy \cite{BR}.
Briot and Ragoucy also looked at certain non-principal
nilpotent orbits which they
assert are connected to
higher rank super Yangians, although we do not understand their
approach.
It should be possible to combine the methods of this article with
those of \cite{BK}
to establish such a
connection
for {\em all} nilpotent orbits
in $\mathfrak{gl}_{m|n}(\C)$.
However this is not trivial and will require
some new presentations for the higher rank super
Yangians adapted to arbitrary parity sequences; the ones
in \cite{G, Peng} are not sufficient as they only apply to the standard parity sequence.
This is currently under investigation by Peng \cite{Peng2}.

By analogy with Kostant's results from \cite{K}
our expectation is that $W_{m|n}$
will play a distinguished role in the
representation theory of
$\mathfrak{gl}_{m|n}(\C)$.
In a forthcoming article \cite{BBG}, we will investigate the {\em Whittaker coinvariants functor} $H_0$, a certain
exact functor
from the analogue of category $\mathcal O$ for $\mathfrak{gl}_{m|n}(\C)$
to the category
of finite dimensional $W_{m|n}$-modules.
We view this as a replacement for
 Soergel's functor $\mathbb{V}$ from
\cite{Soergel};
see also
\cite{Backelin}.
We will show that $H_0$ sends
irreducible modules in $\mathcal O$ to irreducible $W_{m|n}$-modules or zero, and that all
irreducible $W_{m|n}$-modules occur in this way; this should be compared with the analogous result for
parabolic category $\mathcal O$ for
$\mathfrak{gl}_{m+n}(\C)$
obtained in \cite[Theorem E]{BKrep}.
 We will also use properties of $H_0$ to prove that the center of $W_{m|n}$ is
isomorphic to the center of the universal enveloping superalgebra of
$\mathfrak{gl}_{m|n}(\C)$.

\vspace{2mm}
\noindent
{\em Notation.}
We denote the parity of a homogeneous vector $x$ in a
$\Z / 2$-graded vector space
 by $|x| \in \{\0,\1\}$.
A {\em superalgebra} means a $\Z/ 2$-graded
algebra over $\C$.
For homogeneous $x$ and $y$ in an associative superalgebra $A = A_\0
\oplus A_\1$,
their {\em supercommutator} is
$[x,y] \coloneqq xy - (-1)^{|x||y|} yx$.
We say that $A$ is {\em supercommutative} if $[x,y] = 0$
for all homogeneous $x, y \in A$.
Also for homogeneous $x_1,\dots,x_n \in A$,
an {\em ordered supermonomial}
in $x_1,\dots,x_n$
means a monomial of the form $x_1^{i_1} \cdots x_n^{i_n}$
for $i_1,\dots,i_n \geq 0$ such that $i_j \leq 1$ if $x_j$ is odd.

\section{Shifted Yangians}

Recall that $\mathfrak{gl}_{m|n}(\C)$
is the Lie superalgebra of all $(m+n) \times (m+n)$
complex matrices under the supercommutator, with $\Z / 2$-grading
defined so that the matrix unit $e_{i,j}$ is even if $1 \leq i,j \leq m$ or
$m+1 \leq i,j \leq m+n$, and $e_{i,j}$ is odd otherwise.
We denote its universal enveloping superalgebra by
$U(\mathfrak{gl}_{m|n})$; it has basis given by all ordered supermonomials in the matrix units.

The Yangian $Y(\mathfrak{gl}_{m|n})$ was introduced originally by
Nazarov \cite{N}
; see also \cite{G}.
We only need here the special case of $Y = Y(\mathfrak{gl}_{1|1})$.
For its definition we fix
a choice of {\em parity sequence}
\begin{equation}\label{parseq}
(|1|,|2|) \in \Z / 2 \times \Z / 2
\end{equation}
with $|1| \neq |2|$.
All subsequent notation in the remainder of the article
depends implicitly on this choice.
Then we define $Y$
to be the associative
superalgebra on generators $\{t_{i,j}^{(r)}\:|\:1 \leq i,j \leq 2, r
> 0\}$,
with $t_{i,j}^{(r)}$ of parity $|i|+|j|$,
subject to the relations
\begin{align*}
\left[t_{i,j}^{(r)}, t_{p,q}^{(s)}\right]
&= (-1)^{|i| |j| + |i| |p| + |j||p|}
\sum_{a=0}^{\min(r,s)-1}
\left(t_{p,j}^{(a)}
t_{i,q}^{(r+s-1-a)}
-t_{p,j}^{(r+s-1-a)} t_{i,q}^{(a)}\right),
\end{align*}
adopting the convention that $t_{i,j}^{(0)} = \delta_{i,j}$
(Kronecker delta).

\begin{Remark}\label{nodef}
In the literature, one typically only finds results about
$Y(\mathfrak{gl}_{1|1})$
proved for the definition coming from the
parity sequence $(|1|,|2|) = (\0,\1)$.
To aid in translating between this and the
other possibility,
we note that the map
$t_{i,j}^{(r)} \mapsto (-1)^r t_{i,j}^{(r)}$
defines an isomorphism between the
realizations of $Y(\mathfrak{gl}_{1|1})$ arising from the two
choices of parity sequence.
\end{Remark}

As in \cite{N}, we introduce the generating function
$$
t_{i,j}(u) \coloneqq \sum_{r \geq 0} t_{i,j}^{(r)} u^{-r} \in Y\llbracket u^{-1}\rrbracket.
$$
Then
$Y$ is a Hopf superalgebra with comultiplication $\Delta$
and counit $\eps$
given in terms of generating functions by
\begin{align}\label{comult}
\Delta(t_{i,j}(u)) &= \sum_{h=1}^2 t_{i,h}(u) \otimes t_{h,j}(u),\\
\eps(t_{i,j}(u)) &= \delta_{i,j}.\label{counit}
\end{align}
There are also algebra homomorphisms
\begin{align}
\operatorname{in}&:U(\gl_{1|1}) \rightarrow
Y,
&
e_{i,j} &\mapsto (-1)^{|i|}t_{i,j}^{(1)},\\
\ev&:Y \rightarrow
U(\mathfrak{gl}_{1|1}),
&t_{i,j}^{(r)} &\mapsto
\delta_{r,0} \delta_{i,j}+(-1)^{|i|}\delta_{r,1} e_{i,j}.\label{eval}
\end{align}
The composite $\ev \circ \operatorname{in}$
is the identity, hence $\operatorname{in}$ is injective
and $\ev$ is surjective.
We call $\ev$ the {\em evaluation homomorphism}.

We need another set of generators for $Y$ called {\em Drinfeld generators}.
To define these,  we consider the Gauss factorization
$T(u) = F(u) D(u) E(u)$
of the matrix
$$
T(u) \coloneqq
\left(
\begin{array}{ll}
t_{1,1}(u)&t_{1,2}(u)\\
t_{2,1}(u)&t_{2,2}(u)
\end{array}
\right).
$$
This defines
power series $d_i(u), e(u), f(u) \in Y\llbracket u^{-1}
\rrbracket$ such that
$$
D(u) = \left(
\begin{array}{ll}
  d_1(u) & 0 \\0 & d_2(u)
\end{array}
\right),\quad
E(u) = \left(
\begin{array}{ll}
  1 & e(u) \\0 & 1
\end{array}
\right),
\quad
F(u) = \left(
\begin{array}{ll}
  1 & 0 \\f(u) & 1
\end{array}
\right).
$$
Thus we have that
\begin{align}\label{form1}
d_1(u) &= t_{1,1}(u),
& d_2(u) &= t_{2,2}(u) - t_{2,1}(u)t_{1,1}(u)^{-1}t_{1,2}(u),\\
 e(u) &= t_{1,1}(u)^{-1}t_{1,2}(u),&
 f(u) &= t_{2,1}(u) t_{1,1}(u)^{-1}.\label{form2}
\end{align}
Equivalently,
\begin{align}\label{form3}
t_{1,1}(u) &= d_{1}(u),
& t_{2,2}(u) &= d_{2}(u) + f(u) d_1(u) e(u),\\
t_{1,2}(u) &= d_1(u) e(u),&
t_{2,1}(u) &= f(u) d_1(u).\label{form4}
\end{align}
The {Drinfeld generators} are the elements
$d_i^{(r)}, e^{(r)}$ and $f^{(r)}$ of $Y$
defined from the expansions
$d_i(u) = \sum_{r \geq 0} d_i^{(r)} u^{-r}$,
$e(u) = \sum_{r \geq 1} e^{(r)} u^{-r}$ and
$f(u) = \sum_{r \geq 1} f^{(r)} u^{-r}$.
Also define $\tilde d_i^{(r)} \in Y$
from the identity
$\tilde{d}_i(u) = \sum_{r \geq 0} \tilde{d}_i^{(r)}u^{-r} \coloneqq
d_i(u)^{-1}$.

\begin{Theorem}{\cite[Theorem 3]{G}}\label{T:gens}
The superalgebra $Y$ is generated by the even elements
$\{d_i^{(r)}\:|\:i = 1,2, r > 0\}$
and the  odd elements $\{e^{(r)}, f^{(r)}\:|\:r > 0\}$
subject only to the following relations:
\begin{align*}
[d_i^{(r)}, d_j^{(s)}] &= 0,
&[e^{(r)}, f^{(s)}] &= (-1)^{|1|}\sum_{a=0}^{r+s-1} \tilde d_1^{(a)} d_2^{(r+s-1-a)},\\
[e^{(r)}, e^{(s)}] &= 0,
&[d_i^{(r)}, e^{(s)}] &=(-1)^{|1|}
\sum_{a=0}^{r-1} d_i^{(a)} e^{(r+s-1-a)},\\
[f^{(r)}, f^{(s)}] &= 0,
&[d_i^{(r)}, f^{(s)}] &=-(-1)^{|1|}
\sum_{a=0}^{r-1} f^{(r+s-1-a)} d_i^{(a)}.
\end{align*}
Here $d_i^{(0)} = 1$
and $\tilde d_i^{(r)}$ is
defined recursively from
$\sum_{a=0}^r \tilde d_i^{(a)} d_i^{(r-a)} = \delta_{r,0}$.
\end{Theorem}

\begin{Remark}\rm\label{center}
By \cite[Theorem 4]{G} the coefficients
$\{c^{(r)}\:|\:r > 0\}$ of
the power series
$c(u) = \sum_{r \geq 0} c^{(r)} u^{-r} \coloneqq d_1(u) \tilde{d}_2(u)$
generate the center of $Y$.
Hence, so do the coefficients
$\{\tilde c^{(r)}\:|\:r > 0\}$ of
the power series
\begin{equation}\label{tildec}
\tilde c(u) = \sum_{r \geq 0} \tilde c^{(r)} u^{-r} \coloneqq \tilde d_1(u) d_2(u).
\end{equation}
Moreover,
$[e^{(r)}, f^{(s)}] = (-1)^{|1|}\tilde{c}^{(r+s-1)}$, so
these supercommutators are central.
\end{Remark}

\begin{Remark}
Using the relations in Theorem~\ref{T:gens}, one can check that
$Y$
admits
an algebra
 automorphism
\begin{align}\label{zetadef}
\zeta:Y &\rightarrow Y,\qquad
d_1^{(r)} \mapsto
\tilde{d}_2^{(r)},
d_2^{(r)} \mapsto
\tilde{d}_1^{(r)},
e^{(r)} \mapsto
-f^{(r)},
f^{(r)} \mapsto
-e^{(r)}.
\end{align}
By \cite[Proposition 4.3]{G},
this satisfies
\begin{equation}\label{goodthing}
\Delta \circ \zeta =
P \circ (\zeta \otimes \zeta) \circ \Delta
\end{equation}
where
$P(x \otimes y) = (-1)^{|x||y|} y \otimes x$.
\end{Remark}

\begin{Proposition}\label{comultrem}
The comultiplication $\Delta$
is given on Drinfeld generators by the following:
\begin{align*}
\Delta(d_1(u)) &= d_1(u) \otimes d_1(u) + d_1(u) e(u)
\otimes f(u) d_1(u),\\
\Delta(\tilde{d}_1(u)) &=
   \sum_{n \geq 0} (-1) ^{\lceil n/2 \rceil} e(u)^n \tilde d_1(u) \otimes
           \tilde d_1(u) f(u)^n,
\end{align*}\begin{align*}
\Delta(d_2(u)) &= \sum_{n \geq 0}
(-1)^{\lfloor n/2 \rfloor} d_2(u) e(u)^n \otimes f(u)^n d_2(u),\\
\Delta(\tilde{d}_2(u)) &=\tilde d_2(u) \otimes \tilde d_2(u) - e(u)
\tilde d_2(u) \otimes \tilde d_2(u) f(u),
\end{align*}\begin{align*}
\Delta(e(u)) &= 1 \otimes e(u) - \sum_{n \geq 1}
(-1)^{\lceil n/2\rceil}
e(u)^n \otimes \tilde d_1(u) f(u)^{n-1} d_2(u),\\
\Delta(f(u)) &= f(u) \otimes 1 - \sum_{n \geq 1} (-1)^{\lceil n/2 \rceil}
d_2(u)e(u)^{n-1}\tilde d_1(u) \otimes f(u)^n.
\end{align*}
\end{Proposition}

\begin{proof}
Check the formulae for $d_1(u), \tilde d_1(u)$ and $e(u)$
directly using (\ref{comult}) and (\ref{form1})--(\ref{form2}).
The other formulae then follow using (\ref{goodthing}).
\end{proof}

Here is the {\em PBW theorem} for $Y$.

\begin{Theorem}[{\cite[Theorem 1]{G}}]\label{pbw}
Order the set $\{t_{i,j}^{(r)} \:|\: 1 \leq i,j \leq 2, r > 0\}$
in some way.
The ordered supermonomials in these generators
give a basis for $Y$.
\end{Theorem}

There are two important filtrations on $Y$.
First we have the {\em Kazhdan filtration} which is defined
by declaring that the generator
$t_{i,j}^{(r)}$ is in degree $r$, i.e.\ the
filtered degree $r$ part $F_r Y$ of $Y$
with respect to the Kazhdan
filtration is the span of all monomials in the generators
of the form $t_{i_1,j_1}^{(r_1)} \cdots t_{i_n, j_n}^{(r_n)}$
such that $r_1 + \cdots + r_n \leq r$.
The defining relations imply that the
associated graded superalgebra $\gr Y$ is supercommutative.
Let $\mathfrak{gl}_{1|1}[x]$ denote the current Lie superalgebra
$\mathfrak{gl}_{1|1}(\C) \otimes_\C \C[x]$
with basis $\{e_{i,j} x^r\:|\:1 \leq i,j \leq 2, r \geq 0\}$.
Then Theorem~\ref{pbw}
implies that
$\gr Y$
 can be identified with the
symmetric superalgebra $S(\mathfrak{gl}_{1|1}[x])$ of the vector superspace
$\mathfrak{gl}_{1|1}[x]$
so that $\gr_r t_{i,j}^{(r)} = (-1)^{|i|}e_{i,j} x^{r-1}$.

The other filtration on $Y$, which we call the {\em Lie filtration}, is defined similarly by declaring that
$t_{i,j}^{(r)}$ is in degree $r-1$.
In this case we denote the filtered degree $r$ part of $Y$ by
$F'_r Y$ and the associated graded superalgebra by $\gr' Y$.
By Theorem~\ref{pbw} and the defining relations once again,
$\gr' Y$
can be identified with the
universal enveloping superalgebra
$U(\mathfrak{gl}_{1|1}[x])$ so that $\gr_{r-1}' t_{i,j}^{(r)}=(-1)^{|i|}e_{i,j} x^{r-1}$.
The Drinfeld generators
$d_i^{(r)}, e^{(r)}$ and
$f^{(r)}$ all lie in $F_{r-1}' Y$ and we have that
$$
\gr_{r-1}' d_i^{(r)}=
\gr_{r-1}' t_{i,i}^{(r)},
\qquad
\gr_{r-1}' e^{(r)}
=\gr_{r-1}' t_{1,2}^{(r)},
\qquad
\gr_{r-1}' f^{(r)} =
\gr_{r-1}' t_{2,1}^{(r)}.
$$
(The situation for the Kazhdan filtration is more
complicated:
although $d_i^{(r)}, e^{(r)}$ and
$f^{(r)}$ do all lie in $F_r Y$, their images in $\gr_r Y$ are not
in general equal to the images of $t_{i,i}^{(r)}, t_{1,2}^{(r)}$ or
$t_{2,1}^{(r)}$, but they can expressed in terms of them
via (\ref{form1})--(\ref{form2}).)

Combining the preceding discussion of the Lie filtration with
Theorem~\ref{pbw}, we obtain the following
basis for $Y$ in terms of Drinfeld generators.
(One can also deduce this by working with the Kazhdan filtration
and using (\ref{form1})--(\ref{form4}).)

\begin{Corollary} \label{C:pbw}
Order the set $\{d_i^{(r)}\:|\:i=1,2, r > 0\}
\cup \{e^{(r)}, f^{(r)} \:|\: r > 0\}$ in some way.
The ordered supermonomials in these generators
give a basis for $Y$.
\end{Corollary}

Now we are ready to introduce the
{\em shifted Yangians}
for $\mathfrak{gl}_{1|1}(\C)$.
This parallels the definition of shifted Yangians in the purely even
case from
\cite[$\S$2]{BK}.
Let
$\sigma = (s_{i,j})_{1 \leq i,j \leq 2}$
be a $2 \times 2$ matrix of non-negative integers with
$s_{1,1} = s_{2,2} = 0$.
We refer to such a matrix as a {\em shift matrix}.
Let $Y_\sigma$ be the superalgebra
with even
generators
$\{d_i^{(r)}\:|\:i = 1,2, r >0\}$
and odd generators
$\{e^{(r)}\:|\:r > s_{1,2}\} \cup \{f^{(r)}\:|\:r > s_{2,1}\}$
subject to all of the relations from Theorem~\ref{T:gens}
that make sense,
bearing in mind that we no longer have available
the generators $e^{(r)}$ for $0 < r \leq s_{1,2}$ or $f^{(r)}$ for
$0 < r \leq s_{2,1}$.
Clearly there is a homomorphism $Y_\sigma \to Y$ which sends the generators
of $Y_\sigma$ to the generators with the same name in $Y$.

\begin{Theorem}
Order the set
$\{d_i^{(r)}\:|\:i = 1,2, r > 0\}
\cup
\{e^{(r)}\:|\:r > s_{1,2}\} \cup \{f^{(r)}\:|\:r > s_{2,1}\}$
in some way.
The ordered supermonomials in these generators
give a basis for $Y_\sigma$.
In particular, the homomorphism $Y_\sigma \to Y$ is injective.
\end{Theorem}

\begin{proof}
It is easy to see from the defining relations that the monomials span,
and their images in $Y$ are linearly independent by Corollary~\ref{C:pbw}.
\end{proof}

From now on we will identify $Y_\sigma$ with a subalgebra of $Y$
via the
injective homomorphism $Y_\sigma \hookrightarrow Y$.
The Kazhdan and Lie filtrations on $Y$ induce filtrations on
$Y_\sigma$
such that $\gr Y_\sigma \subseteq \gr Y$ and $\gr' Y_\sigma
\subseteq \gr' Y$.
Let $\mathfrak{gl}_{1|1}^\sigma[x]$ be the Lie subalgebra of
$\mathfrak{gl}_{1|1}[x]$ spanned by the vectors
$e_{i,j} x^r$ for $1 \leq i,j \leq 2$ and
$r \geq s_{i,j}$.
Then we have that
$\gr Y_\sigma = S(\mathfrak{gl}_{1|1}^\sigma[x])$
and
$\gr' Y_\sigma = U(\mathfrak{gl}_{1|1}^\sigma[x])$.

\begin{Remark}\label{iotasr}
Given another shift matrix $\sigma' = (s_{i,j}')_{1\leq i,j
  \leq 2}$ with $s_{2,1}'+s_{1,2}' = s_{2,1}+s_{1,2}$ there is an
isomorphism
\begin{equation}\label{iotas}
\iota:Y_\sigma \stackrel{\sim}{\rightarrow} Y_{\sigma'},
\quad
d_i^{(r)}\mapsto d_i^{(r)},
e^{(r)} \mapsto e^{(s_{1,2}'-s_{1,2}+r)},
f^{(r)} \mapsto f^{(s_{2,1}'-s_{2,1}+r)}.
\end{equation}
This follows from the defining relations.
Thus, up to isomorphism, $Y_\sigma$
depends only on the integer $s_{2,1}+s_{1,2} \geq 0$, not on $\sigma$
itself.
Beware though that the isomorphism $\iota$
does not respect the Kazhdan or Lie filtrations.
\end{Remark}

For $\sigma \neq 0$, $Y_\sigma$ is not a Hopf subalgebra of
$Y$. However
there are
some useful comultiplication-like homomorphisms between different shifted
Yangians.
To start with, let
$\sigma^{\up}$ (resp.\ $\sigma^{\lo}$) be the upper (resp.\ lower) triangular
shift  matrix obtained from $\sigma$
by setting $s_{2,1}$ (resp.\ $s_{1,2}$) equal to zero.
Then, by Proposition~\ref{comultrem}, the restriction of the comultiplication $\Delta$ on $Y$
gives a homomorphism
\begin{align}\label{gin}
\Delta: Y_\sigma \rightarrow Y_{\sigma^{\lo}} \otimes
Y_{\sigma^{\up}}.
\end{align}
The remaining comultiplication-like homomorphisms involve the
universal enveloping algebra $U(\mathfrak{gl}_1) = \C[e_{1,1}]$.
Assuming that $s_{1,2} > 0$,
let $\sigma_+$ be the shift matrix obtained from $\sigma$ by
subtracting $1$ from the entry $s_{1,2}$.
Then the relations imply that there is a well-defined
algebra homomorphism
\begin{align}\label{dp}
\Delta_+:Y_\sigma &\rightarrow Y_{\sigma_+} \otimes U(\mathfrak{gl}_1),\\
d_1^{(r)} &\mapsto d_1^{(r)} \otimes 1,&
d_2^{(r)} &\mapsto d_2^{(r)}\otimes 1 +(-1)^{|2|} d_2^{(r-1)}
 \otimes  e_{1,1},
\notag\\
 e^{(r)} &\mapsto e^{(r)} \otimes 1 + (-1)^{|2|}e^{(r-1)} \otimes e_{1,1},\!\!
&f^{(r)} &\mapsto f^{(r)} \otimes 1.\notag
\end{align}
Finally, assuming that
$s_{2,1} > 0$, let $\sigma_-$ be the shift
matrix
obtained from $\sigma$ by subtracting $1$ from $s_{2,1}$.
Then there is an algebra homomorphism
\begin{align}\label{dm}
\Delta_-:Y_{\sigma} &\rightarrow U(\mathfrak{gl}_1) \otimes Y_{\sigma_-},\\
d_1^{(r)} &\mapsto 1 \otimes d_1^{(r)},&
d_2^{(r)} &\mapsto 1 \otimes d_2^{(r)} +(-1)^{|2|} e_{1,1} \otimes d_2^{(r-1)},
\notag\\
f^{(r)} &\mapsto 1 \otimes f^{(r)} +(-1)^{|2|}e_{1,1} \otimes
f^{(r-1)},\!\!
& e^{(r)} &\mapsto 1\otimes e^{(r)}.\notag
\end{align}
If $s_{1,2} > 0$, we denote
$(\sigma^{\up})_+ = (\sigma_+)^{\up}$ by
$\sigma^\up_+$.
If $s_{2,1} > 0$ we
denote $(\sigma^{\lo})_- = (\sigma_-)^{\lo}$ by $\sigma^\lo_-$.
If both $s_{1,2} > 0$ and $s_{2,1} > 0$ we
denote $(\sigma_+)_- = (\sigma_-)_+$ by $\sigma_{\pm}$.

\begin{Lemma}\label{coa}
Assuming that $s_{1,2} > 0$ in the first diagram,
$s_{2,1} > 0$ in the second diagram,
and both $s_{1,2} > 0$ and $s_{2,1} > 0$ in the final diagram,
the following commute:
\begin{align}\label{ca1}
&\begin{CD}
Y_\sigma&@>\Delta_+>>&Y_{\sigma_+} \otimes U(\mathfrak{gl}_1)\\
@V\Delta VV&&@VV\Delta \otimes \id V\\
Y_{\sigma^\lo} \otimes Y_{\sigma^\up}
&@>\id \otimes \Delta_+>>&Y_{\sigma^\lo} \otimes Y_{\sigma^\up_+} \otimes
U(\mathfrak{gl}_1)
\end{CD}\end{align}\begin{align}\label{ca2}
&\begin{CD}
Y_\sigma&@>\Delta>>&Y_{\sigma^\lo} \otimes Y_{\sigma^\up}\\
@V\Delta_-VV&&@VV\Delta_-\otimes\id V\\
U(\mathfrak{gl}_1) \otimes
Y_{\sigma_-}
&@>\id \otimes \Delta>>&U(\mathfrak{gl}_1) \otimes Y_{\sigma^\lo_-} \otimes Y_{\sigma^\up}
\end{CD}\end{align}\begin{align}\label{ca3}
&\begin{CD}
Y_\sigma&@>\Delta_+>>&Y_{\sigma_+} \otimes U(\mathfrak{gl}_1)\\
@V\Delta_-VV&&@VV\Delta_-\otimes\id V\\
U(\mathfrak{gl}_1) \otimes
Y_{\sigma_{-}}
&@>\id \otimes \Delta_+>>&U(\mathfrak{gl}_1) \otimes Y_{\sigma_{\pm}} \otimes U(\mathfrak{gl}_1)
\end{CD}
\end{align}
\end{Lemma}

\begin{proof}
Check on Drinfeld generators using (\ref{dp})--(\ref{dm})
and Proposition~\ref{comultrem}.
\end{proof}

\begin{Remark}
Writing $\eps:U(\mathfrak{gl}_1) \rightarrow \C$
for the counit, the maps
$(\id \bar\otimes\eps) \circ \Delta_+$
and $(\eps \bar\otimes\id) \circ \Delta_-$
are the natural inclusions
$Y_\sigma\rightarrow Y_{\sigma_+}$
and
$Y_\sigma\rightarrow Y_{\sigma_-}$,
respectively. Hence the maps $\Delta_+$ and $\Delta_-$
are injective.
\end{Remark}

\section{Truncation}

Let $\sigma = (s_{i,j})_{1 \leq i,j \leq 2}$ be a shift matrix.
Suppose also that we are given an integer $l \geq s_{2,1}+s_{1,2}$
and set
$$
k \coloneqq l - s_{2,1}-s_{1,2} \geq 0.
$$
In view of Lemma~\ref{coa},
we can iterate
$\Delta_+$
a total of
$s_{1,2}$ times,
$\Delta_-$
a total of $s_{2,1}$ times, and
$\Delta$ a total of $(k-1)$ times
in any order that makes sense
(when $k=0$ this means we apply the
counit $\eps$ once at the very end) to obtain a well-defined homomorphism
$$
\Delta^{l}_{\sigma}:Y_\sigma \to U(\mathfrak{gl}_1)^{\otimes s_{2,1}}
\otimes Y^{\otimes k} \otimes U(\mathfrak{gl}_1)^{\otimes s_{1,2}}.
$$
For example, if $\sigma =
\left(\begin{array}{ll}0&2\\1&0\end{array}\right)$
then
\begin{align*}
\Delta^{3}_\sigma &= (\id\otimes \eps\bar\otimes\id\otimes\id)\circ
(\Delta_-\otimes\id\otimes\id)\circ(\Delta_+ \otimes \id) \circ \Delta_+,\\
\Delta^{4}_\sigma &
=
(\id \otimes \Delta_+\otimes\id)\circ(\Delta_- \otimes \id) \circ \Delta_+
=
(\id \otimes \Delta_+\otimes\id)\circ(\id \otimes \Delta_+) \circ \Delta_-,\\
\Delta^{5}_\sigma
&=
(\Delta_-\otimes\id\otimes\id\otimes \id)\circ(\id \otimes \Delta_+ \otimes \id) \circ
(\id \otimes \Delta_+)\circ
\Delta\\
&=
(\id \otimes \Delta\otimes\id\otimes \id)\circ(\Delta_- \otimes\id\otimes \id) \circ
(\id \otimes \Delta_+)\circ
\Delta_+.
\end{align*}
Let
\begin{equation}\label{usl}
U^{l}_\sigma \coloneqq
U(\mathfrak{gl}_1)^{\otimes s_{2,1}} \otimes
U(\mathfrak{gl}_{1|1})^{\otimes k} \otimes U(\mathfrak{gl}_1)^{\otimes
  s_{1,2}},
\end{equation}
viewed as a superalgebra using
the usual sign convention.
Recalling (\ref{eval}), we obtain a homomorphism
\begin{equation}\label{Evalhom}
\ev_\sigma^l \coloneqq (\id^{\otimes s_{2,1}} \otimes \ev^{\otimes k} \otimes
\id^{\otimes s_{1,2}}) \circ \Delta^{l}_{\sigma}:Y_\sigma
\to U^{l}_\sigma.
\end{equation}
Let
\begin{equation}\label{shy}
Y_\sigma^{l} \coloneqq \ev_\sigma^l(Y_\sigma) \subseteq U_\sigma^{l}.
\end{equation}
This is the {\em shifted Yangian of level $l$}.

In the special case that $\sigma = 0$ we denote
$\ev_\sigma^l$, $Y_\sigma^{l}$ and $U_\sigma^l$
simply
by $\ev^l$, $Y^{l}$ and $U^l$, respectively, so that
$Y^l = \ev^l(Y) \subseteq U^l$.
We call $Y^l$ the {\em Yangian of level $l$}.
Writing
$\bar e_{i,j}^{[c]} \coloneqq
(-1)^{|i|}1^{\otimes (c-1)} \otimes e_{i,j} \otimes 1^{\otimes(l-c)}$,
we have simply that
\begin{equation}\label{new}
\ev^l(t_{i,j}^{(r)})
=
\sum_{1 < c_1 < \cdots < c_r \leq l}
\sum_{1 \leq h_1,\dots,h_{r-1} \leq 2}
\bar e_{i,h_1}^{[c_1]}
\bar e_{h_1,h_2}^{[c_2]}
\cdots
\bar e_{h_{r-1},j}^{[c_r]}
\end{equation}
for any $1 \leq i,j \leq 2$ and $r \geq 0$.
In particular, $\ev^l(t_{i,j}^{(r)}) = 0$ for $r > l$.
In the proof of \cite[Theorem 1]{G}, Gow shows
that the kernel of $\ev^l:Y \twoheadrightarrow Y^{l}$
is generated by $\{t_{i,j}^{(r)}\:|\:1 \leq i,j \leq 2,
r > l\}$, and moreover the images of the ordered supermonomials in
the remaining elements $\{t_{i,j}^{(r)}\:|\:1 \leq i,j \leq 2, 0 < r
\leq l\}$ give a basis for $Y^{l}$.
(Actually she proves this for all $Y(\mathfrak{gl}_{m|n})$ not just
$Y(\mathfrak{gl}_{1|1})$.)
The goal in this section is to prove analogues of
these statements for $Y_\sigma$
with $\sigma \neq 0$.

Let $I_\sigma^{l}$ be the two-sided ideal of $Y_\sigma$
generated by the elements $d_1^{(r)}$ for $r > k$.

\begin{Lemma}\label{fir}
$I_\sigma^{l} \subseteq \ker \ev_\sigma^l$.
\end{Lemma}

\begin{proof}
We need to show that $\ev_\sigma^l(d_1^{(r)}) = 0$ for all $r > k$.
We calculate this by first
applying all the maps
$\Delta_+$ and $\Delta_-$
to deduce that
$$
\ev_\sigma^l (d_1^{(r)}) =
1^{\otimes s_{2,1}} \otimes
\ev^k(d_1^{(r)})
\otimes 1^{\otimes s_{1,2}}.
$$
Since $d_1^{(r)} =t_{1,1}^{(r)}$,
it is then clear from (\ref{new})
that $\ev^k(d_1^{(r)}) = 0$
for $r > k$.
\end{proof}

\begin{Proposition} \label{L:hzero}
The ideal $I_\sigma^{l}$ contains all of the following elements:
\begin{align}
\sum_{s_{1,2} < a \leq r} d_1^{(r-a)} e^{(a)}&\text{ for $r > s_{1,2}+k$;}\label{van1}\\
\sum_{s_{2,1} < b \leq r} f^{(b)} d_1^{(r-b)}&\text{ for $r > s_{2,1}+k$;}\label{van2}\\
 d_2^{(r)} + \sum_{\substack{s_{1,2} < a\\s_{2,1} < b\\ a + b \leq r}}
f^{(b)} d_1^{(r-a-b)} e^{(a)}&\text{ for $r >l$.}\label{van3}
\end{align}
\end{Proposition}

\begin{proof}
Consider the algebra
$Y_\sigma\llbracket u^{-1} \rrbracket [u]$ of formal Laurent series in the variable
$u^{-1}$ with coefficients in $Y_\sigma$.
For any such formal Laurent series $p = \sum_{r \leq N} p_r
u^{r}$
we write $[p]_{\geq 0}$ for its polynomial part
$\sum_{r =0}^N p_{r} u^{r}$.
Also write $\equiv$ for congruence
modulo $Y_\sigma[u] + u^{-1}I_\sigma^{l} \llbracket u^{-1}
\rrbracket$,
so $p \equiv 0$ means that the
$u^{r}$-coefficients of $p$ lie in $I_\sigma^{l}$ for all $r < 0$.
Note that if $p \equiv 0, q \in Y_\sigma[u]$, then $pq \equiv 0$.
In this notation, we have by definition of $I_\sigma^{l}$
that $u^{k} d_1(u) \equiv 0$.
Introduce the power series
\begin{align*}
e_\sigma(u) &\coloneqq \sum_{r > s_{1,2}} e^{(r)}
u^{-r},&
f_\sigma(u) &\coloneqq \sum_{r > s_{2,1}} f^{(r)} u^{-r}.
\end{align*}
The proposition is equivalent to the following assertions:
\begin{align}\label{id1}
u^{s_{1,2}+k}d_1(u) e_\sigma(u) &\equiv 0,\\\label{id2}
u^{s_{2,1}+k}f_\sigma(u) d_1(u) &\equiv 0,\\
u^{l} \left(d_2(u) + f_\sigma(u) d_1(u) e_\sigma(u)\right)
&\equiv 0.\label{id3}
\end{align}
For the first two, we use the identities
\begin{align}
(-1)^{|1|} [d_1(u), e^{(s_{1,2}+1)}] &= u^{s_{1,2}} d_1(u)
e_\sigma(u),\label{new1}\\
 (-1)^{|1|} [f^{(s_{2,1}+1)}, d_1(u)] &= u^{s_{2,1}} f_\sigma(u)
d_1(u).\label{new2}
\end{align}
These
are easily checked by considering the $u^{-r}$-coefficients on
each side and using the relations in Theorem~\ref{T:gens}.
Assertions (\ref{id1})--(\ref{id2}) follow from
(\ref{new1})--(\ref{new2})
on multiplying by $u^k$ as $u^k d_1(u) \equiv 0$.
For the final assertion (\ref{id3}),
recall the elements $\tilde c^{(r)}$ from (\ref{tildec}).
Let
$\tilde c_\sigma(u) \coloneqq \sum_{r > s_{2,1}+s_{1,2}} \tilde c^{(r)}
u^{-r}.$
Another routine check using the relations shows that
\begin{equation}\label{new3}
(-1)^{|1|} [f^{(s_{2,1}+1)}, e_\sigma(u)] = u^{s_{2,1}} \tilde
c_\sigma(u).
\end{equation}
Using (\ref{id1}) and (\ref{new2})--(\ref{new3})
we deduce that
\begin{align*}
0 &\equiv
(-1)^{|1|} u^{s_{1,2}+k} [f^{(s_{2,1}+1)}, d_1(u) e_\sigma(u)]\\
&=
u^{s_{1,2}+k} d_1(u) (-1)^{|1|}[f^{(s_{2,1}+1)}, e_\sigma(u)]
+u^{s_{1,2}+k} (-1)^{|1|}[f^{(s_{2,1}+1)}, d_1(u) ]e_\sigma(u)\\
&=
u^{l} d_1(u)
\tilde c_\sigma(u)
+
u^{l} f_\sigma(u) d_1(u)e_\sigma(u).
\end{align*}
To complete the proof of (\ref{id3}), it remains to observe that
$$
u^{s_{2,1}+s_{1,2}}
\tilde c_\sigma(u)
= u^{s_{2,1} + s_{1,2}} \tilde d_1(u) d_2(u)
-[u^{s_{2,1}+s_{1,2}}
\tilde d_1(u) d_2(u)]_{\geq 0}
$$
hence
$u^{l} d_1(u)
\tilde c_\sigma(u)
\equiv u^{l} d_2(u).
$
\end{proof}

For the rest of the section, we fix some total ordering
on the set
\begin{multline}\label{supermonomial}
\Omega \coloneqq \{d_1^{(r)}\:|\:0<r\leq k\} \cup
\{d_2^{(r)}\:|\:0 < r \leq l\}\\
\cup \{e^{(r)}\:|\:s_{1,2} < r \leq s_{1,2}+k\}
\cup \{f^{(r)}\:|\:s_{2,1} < r \leq s_{2,1}+k\}.
\end{multline}

\begin{Lemma}\label{span}
The quotient algebra $Y_\sigma / I_\sigma^{l}$ is spanned by
the images of the
ordered supermonomials in the elements of $\Omega$.
\end{Lemma}

\begin{proof}
The Kazhdan filtration on $Y_\sigma$ induces a filtration on
$Y_\sigma / I_\sigma^{l}$ with respect to which
$\gr (Y_\sigma / I_\sigma^{l})$ is a graded quotient of
$\gr Y_\sigma$.
We know already that
$\gr Y_\sigma$ is supercommutative,
hence so too is
$\gr (Y_\sigma / I_\sigma^{l})$.
Let $\underline{d}_i^{(r)} \coloneqq \gr_r (d_i^{(r)} + I_\sigma^{l}),
\underline{e}^{(r)} \coloneqq \gr_r (e^{(r)}+ I_\sigma^{l})$ and
$\underline{f}^{(r)} \coloneqq \gr_r (f^{(r)}+I_\sigma^{l})$.

To prove the lemma it is enough to show that
$\gr (Y_\sigma / I_\sigma^{l})$
is generated by
$\{\underline{d}_1^{(r)}\:|\:0<r\leq k\} \cup
\{\underline{d}_2^{(r)}\:|\:0 < r \leq l\}
\cup \{\underline{e}^{(r)}\:|\:s_{1,2} < r \leq s_{1,2}+k\}
\cup \{\underline{f}^{(r)}\:|\:s_{2,1} < r \leq s_{2,1}+k\}$.
This follows because
$\underline{d}_1^{(r)} = 0$ for $r > k$, and
each
of the elements
$\underline{d}_2^{(r)}\:(r > l)$,
$\underline{e}^{(r)}\:(r > s_{1,2}+k)$ and
$\underline{f}^{(r)}\:(r > s_{2,1}+k)$
can be expressed as polynomials in generators of strictly
smaller degrees by Proposition~\ref{L:hzero}.
\end{proof}

\begin{Lemma}\label{ind}
The image under $\ev_\sigma^l$ of the ordered supermonomials in the
elements of $\Omega$
are linearly independent in $Y_\sigma^{l}$.
\end{Lemma}

\begin{proof}
Consider the standard filtration on
$U_\sigma^{l}$ generated by declaring that all the
elements of the form $1 \otimes\cdots\otimes 1 \otimes
x \otimes 1 \otimes\cdots\otimes 1$ for $x \in \mathfrak{gl}_1$
or $\mathfrak{gl}_{1|1}$ are in degree $1$.
It induces a filtration on $Y_\sigma^{l}$
so that $\gr Y_\sigma^{l}$ is a graded subalgebra of $\gr
U_\sigma^{l}$.
Note that $\gr U_\sigma^{l}$ is supercommutative, hence so is
the subalgebra $\gr Y_\sigma^{l}$.
Each of the elements $\ev_\sigma^l(d_i^{(r)}), \ev_\sigma^l(e^{(r)})$ and $\ev_\sigma^l(f^{(r)})$
are in filtered degree $r$ by the definition of
$\ev_\sigma^l$.
Let $\underline{d}_i^{(r)} \coloneqq \gr_r(\ev_\sigma^l(d_i^{(r)})),
\underline{e}^{(r)} \coloneqq \gr_r(\ev_\sigma^l(e^{(r)}))$ and
$\underline{f}^{(r)} \coloneqq \gr_r(\ev_\sigma^l(f^{(r)}))$.

Let $M$ be the set
of ordered supermonomials in
$\{\underline{d}_1^{(r)}\:|\:0<r\leq k\} \cup
\{\underline{d}_2^{(r)}\:|\:0 < r \leq l\}
\cup \{\underline{e}^{(r)}\:|\:s_{1,2} < r \leq s_{1,2}+k\}
\cup \{\underline{f}^{(r)}\:|\:s_{2,1} < r \leq s_{2,1}+k\}$.
To prove the lemma, it suffices to show that $M$
is linearly independent in $\gr Y_\sigma^{l}$.
For this,
we proceed by induction on $s_{2,1}+s_{1,2}$.

To establish the base case $s_{2,1}+s_{1,2} = 0$, i.e.\ $\sigma = 0$,
$Y_\sigma = Y$ and $Y_\sigma^{l} = Y^{l}$.
Let $\underline{t}_{i,j}^{(r)}$ denote $\gr_r (\ev_\sigma^l(t_{i,j}^{(r)}))$.
Fix a total order on
$\{\underline{t}_{i,j}^{(r)}\:|\: 1 \leq i,j \leq 2, 0 < r \leq l\}$ and
let $M'$ be the resulting
set of ordered supermonomials.
Exploiting the explicit formula (\ref{new}),
Gow shows in the proof of \cite[Theorem 1]{G} that $M'$ is linearly
independent.
By (\ref{form1})--(\ref{form4}), any element of $M$ is a linear
combination of elements of $M'$ of the same degree,
and vice versa. So we deduce that $M$ is linearly independent too.

For the induction step, suppose that $s_{2,1}+s_{1,2} > 0$.
Then we either have $s_{2,1} > 0$ or $s_{1,2} > 0$.
We just explain the argument for the latter case; the proof in the former
case is entirely similar replacing $\Delta_+$ with $\Delta_-$.
Recall that $\sigma_+$ denotes the shift matrix obtained from $\sigma$
by subtracting $1$ from $s_{1,2}$.
So $U_{\sigma}^{l} = U_{\sigma_+}^{l-1} \otimes U(\mathfrak{gl}_1)$.
By its definition,
we have that $\ev_\sigma^l
= (\ev^{l-1}_{\sigma_+} \otimes \id) \circ \Delta_+$,
hence
$Y_\sigma^{l} \subseteq
Y_{\sigma_+}^{l-1} \otimes U(\mathfrak{gl}_1)$.
Let $x \coloneqq \gr_1 e_{1,1} \in \gr U(\mathfrak{gl}_1)$.
Then
\begin{align*}
\underline{d}_1^{(r)} &= \dot{\underline{d}}_1^{(r)} \otimes 1,&
\underline{d}_2^{(r)} &= \dot{\underline{d}}_2^{(r)} \otimes 1 + (-1)^{|2|}\dot{\underline{d}}_2^{(r-1)} \otimes x,\\
\underline{f}^{(r)} &= \dot{\underline{f}}^{(r)} \otimes 1,
&\underline{e}^{(r)} &= \dot{\underline{e}}^{(r)} \otimes 1 + (-1)^{|2|}\dot{\underline{e}}^{(r-1)} \otimes x.
\end{align*}
The notation is potentially confusing here so we have decorated
elements of
 $\gr Y_{\sigma_+}^{l-1} \subseteq \gr U_{\sigma_+}^{l-1}$
with a dot.
It remains to observe
from the
induction hypothesis applied to $\gr Y_{\sigma_+}^{l-1}$
that ordered supermonomials in
$\{\dot{\underline{d}}_1^{(r)}\otimes 1\:|\:0 < r \leq k\}
\cup \{\dot{\underline{d}}_2^{(r-1)}\otimes x\:|\:0 < r \leq l\}
\cup \{\dot{\underline{e}}^{(r-1)} \otimes x\:|\:s_{1,2} < r \leq s_{1,2}+k\}
\cup \{\dot{\underline{f}}^{(r)}\otimes 1\:|\:0 < r < s_{1,2}+k\}$
are linearly independent.
\end{proof}

\begin{Theorem}\label{truncated}
The kernel of $\ev_\sigma^l:Y_\sigma\rightarrow Y_\sigma^{l}$
is equal to the two-sided ideal $I_\sigma^{l}$
generated by the elements $\{d_1^{(r)}\:|\:r > k\}$.
Hence $\ev_\sigma^l$ induces an algebra isomorphism
between $Y_\sigma / I_\sigma^{l}$ and
$Y_\sigma^{l}$.
\end{Theorem}

\begin{proof}
By Lemma~\ref{fir}, $\ev_\sigma^l$ induces a surjection
$Y_\sigma / I_\sigma^{l} \twoheadrightarrow Y_\sigma^{l}$.
It maps the spanning set from Lemma~\ref{span}
onto the linearly independent set from Lemma~\ref{ind}.
Hence it is an isomorphism and both sets are actually bases.
\end{proof}

Henceforth we will {\em identify} $Y_\sigma^{l}$ with
the quotient $Y_\sigma / I_\sigma^{l}$, and
we will abuse notation by denoting the canonical images in $Y_\sigma^{l}$
of the elements
$d_i^{(r)}, e^{(r)} \dots$ of $Y_\sigma$
by the same symbols $d_i^{(r)}, e^{(r)}, \dots$.
This will not cause any confusion as we will not work
with $Y_\sigma$ again.

Here is the PBW theorem for $Y_\sigma^{l}$, which was noted already in the
proof of Theorem~\ref{truncated}.

\begin{Corollary}\label{pbwt}
Order
the set
$\{d_1^{(r)}\:|\:0<r\leq k\} \cup
\{d_2^{(r)}\:|\:0 < r \leq l\}
\cup \{e^{(r)}\:|\:s_{1,2} < r \leq s_{1,2}+k\}
\cup \{f^{(r)}\:|\:s_{2,1} < r \leq s_{2,1}+k\}$
in some way.
The ordered supermonomials in these
elements
give a basis for $Y_\sigma^{l}$.
\end{Corollary}

\begin{Remark}
In the arguments in this section, we have defined {\em two}
filtrations on
$Y_\sigma^{l}$, one in the proof of Lemma~\ref{span}
induced by the Kazhdan filtration on $Y_\sigma$, the other in the
proof of
Lemma~\ref{ind} induced by the standard filtration on
$U_\sigma^{l}$.
Using Corollary~\ref{pbwt}, one can check that these two filtrations
coincide.
\end{Remark}

\section{Principal $W$-algebras}\label{spyr}

We turn to the $W$-algebra side of the story.
Let $\pi$ be a (two-rowed) {\em pyramid}, that is,
a collection of boxes
in the plane arranged in two connected rows
such that each box in the
first (top) row lies directly above a box in the second (bottom) row.
For example, here are all the pyramids with two boxes in the first row
and five in the second:
$$
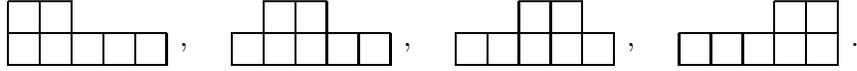
\begin{array}{c}
\begin{picture}(60,24) \put(0,0){\line(1,0){60}}
\put(0,12){\line(1,0){60}} \put(0,24){\line(1,0){24}}
\put(0,0){\line(0,1){24}} \put(12,0){\line(0,1){24}}
\put(24,0){\line(0,1){24}} \put(36,0){\line(0,1){12}}
\put(48,0){\line(0,1){12}}
\put(60,0){\line(0,1){12}}
\end{picture}
\end{array}
,\:\:\:\:
\begin{array}{c}
\begin{picture}(60,24) \put(0,0){\line(1,0){60}}
\put(0,12){\line(1,0){60}} \put(12,24){\line(1,0){24}}
\put(0,0){\line(0,1){12}} \put(12,0){\line(0,1){24}}
\put(24,0){\line(0,1){24}} \put(36,0){\line(0,1){24}}
\put(48,0){\line(0,1){12}}
\put(60,0){\line(0,1){12}}
\end{picture}
\end{array},\:\:\:\:
\begin{array}{c}
\begin{picture}(60,24) \put(0,0){\line(1,0){60}}
\put(0,12){\line(1,0){60}} \put(24,24){\line(1,0){24}}
\put(0,0){\line(0,1){12}} \put(12,0){\line(0,1){12}}
\put(24,0){\line(0,1){24}} \put(36,0){\line(0,1){24}}
\put(48,0){\line(0,1){24}}
\put(60,0){\line(0,1){12}}
\end{picture}\end{array}
,\:\:\:\:
\begin{array}{c}
\begin{picture}(60,24) \put(0,0){\line(1,0){60}}
\put(0,12){\line(1,0){60}} \put(36,24){\line(1,0){24}}
\put(0,0){\line(0,1){12}} \put(12,0){\line(0,1){12}}
\put(24,0){\line(0,1){12}} \put(36,0){\line(0,1){24}}
\put(48,0){\line(0,1){24}}
\put(60,0){\line(0,1){24}}
\end{picture}
\end{array}.
$$
Let $k$ (resp.\ $l$) denote
the number of boxes in the first (resp.\ second) row of $\pi$,
so that $k \leq l$.
The parity sequence fixed in (\ref{parseq}) allows us to talk about the
parities of the rows of $\pi$: the $i$th row is
of parity $|i|$.
Let $m$ be the number of boxes in the even row, i.e.\ the row with
parity $\0$, and $n$ be the number of boxes in the odd row, i.e.\ the
row with parity $\1$.
Then label the boxes in the even (resp.\ odd)
row from left to right by the numbers
$1,\dots,m$ (resp.\ $m+1,\dots,m+n$).
For example
here is one of the above pyramids with boxes labelled in this way
assuming that $(|1|,|2|) = (\1,\0)$, i.e.\ the bottom row is even and
the top row is odd:
\begin{equation}\label{piex}
\begin{array}{c}
\begin{picture}(60,24) \put(0,0){\line(1,0){60}}
\put(0,12){\line(1,0){60}} \put(12,24){\line(1,0){24}}
\put(0,0){\line(0,1){12}} \put(12,0){\line(0,1){24}}
\put(24,0){\line(0,1){24}} \put(36,0){\line(0,1){24}}
\put(48,0){\line(0,1){12}}
\put(60,0){\line(0,1){12}}
\put(18,18){\makebox(0,0){{6}}} \put(30,18){\makebox(0,0){{7}}}
\put(6,6){\makebox(0,0){{1}}}
\put(18,6){\makebox(0,0){{2}}}
\put(30,6){\makebox(0,0){{3}}}
\put(42,6){\makebox(0,0){{4}}}
\put(54,6){\makebox(0,0){{5}}}
\end{picture}
\end{array}.
\end{equation}
Numbering the columns of $\pi$ by $1,\dots,l$ in order from left to
right, we write $\row(i)$ and $\col(i)$ for the row and column numbers
of the $i$th box in this labelling.

Now let $\mathfrak{g} \coloneqq \mathfrak{gl}_{m|n}(\C)$ for $m$ and $n$
coming from the pyramid $\pi$ and the fixed parity sequence
as in the previous paragraph.
Let $\mathfrak{t}$ be the Cartan subalgebra consisting of all diagonal
matrices and $\eps_1,\dots,\eps_{m+n} \in \mathfrak{t}^*$
be the basis such that $\eps_i(e_{j,j}) =
\delta_{i,j}$ for each $j=1,\dots,m+n$.
The {\em supertrace form}
$(.|.)$ on $\g$ is  the
non-degenerate invariant supersymmetric bilinear form
defined from $(x|y) =
\operatorname{str}(xy)$, where the supertrace $\operatorname{str} A$ of matrix
$A = (a_{i,j})_{1 \leq i,j \leq m+n}$
means $a_{1,1}+\cdots+a_{m,m}-a_{m+1,m+1}-\cdots-a_{m+n,m+n}$.
It induces a
bilinear form $(.|.)$ on $\mathfrak{t}^*$
such that
$(\eps_i|\eps_j) = (-1)^{|\row(i)|} \delta_{i,j}$.

We have the explicit principal nilpotent element
\begin{equation}
e \coloneqq \sum_{i,j}e_{i,j}
\in \g_\0
\end{equation}
summing over all adjacent pairs
$\Diagram{ $\scriptstyle i$ & $\scriptstyle j$ \cr}$
of boxes in the pyramid $\pi$.
In the example above, we
have that $e = e_{1,2} + e_{2,3} + e_{3,4} + e_{4,5} + e_{6,7}$.
Let $\chi \in \g^*$ be defined by $\chi(x) \coloneqq (x|e)$.
If we set
\begin{equation}\label{bare}
\bar e_{i,j} \coloneqq (-1)^{|\row(i)|}e_{i,j},
\end{equation}
then we have that
\begin{equation}
\chi(\bar e_{i,j}) =
\begin{cases} 1 & \text{if
$\Diagram{ $\scriptstyle j$ & $\scriptstyle i$ \cr}$
is an adjacent pair of boxes in $\pi$,} \\
0 & \text{otherwise}.
\end{cases}
\end{equation}
Introduce a $\Z$-grading
$\g = \bigoplus_{r \in \Z} \g(r)$
by declaring that $e_{i,j}$ is of degree
\begin{equation}\label{dege}
\deg(e_{i,j}) \coloneqq \col(j)- \col(i).
\end{equation}
This is a {\em good grading} for $e$, which means that $e \in \g(1)$
and the centralizer $\g^e$ of $e$ in $\g$
is contained in $\bigoplus_{r \ge 0} \g(r)$; see \cite{H} for
more about good gradings on Lie superalgebras (one should double the degrees
of our grading to agree with the terminology there).
Set
$$
\p \coloneqq \bigoplus_{r \ge 0} \g(r) \qquad \h \coloneqq \g(0),\qquad
\m \coloneqq\bigoplus_{r < 0} \g(r).
$$
Note that $\chi$ restricts to a character of $\m$.
Let
$\m_\chi \coloneqq \{x-\chi(x) \:|\: x \in \m\}$,
which is a shifted copy of $\mathfrak{m}$ inside $U(\mathfrak{m})$.
Then the {\em principal $W$-algebra} associated to the pyramid $\pi$
is
\begin{equation}\label{wpidef}
W_\pi \coloneqq \{u \in U(\p) \:|\:  u \m_\chi \subseteq \m_\chi U(\g) \}.
\end{equation}
It is straightforward to check that $W_\pi$ is a subalgebra of $U(\p)$.

The first important result
about $W_\pi$ is its {\em PBW theorem}.
This is noted already in
\cite[Remark 3.10]{Zhao}, where it is described for arbitrary basic
classical Lie superalgebras modulo a mild assumption on $e$ (which is
trivially satisfied here).
To formulate the result precisely,
embed $e$ into an $\mf{sl}_2$-triple $(e,h,f)$
 in $\g_\0$, such that $h \in \g(0)$ and $f \in \g(-1)$.
It follows from $\mf{sl}_2$ representation theory
that
\begin{equation} \label{e:psum}
\p = \g^e \oplus [\p^\perp,f],
\end{equation}
where $\p^\perp = \bigoplus_{r > 0} \g(r)$ denotes the nilradical of $\p$.
Also introduce
the {\em Kazhdan filtration} on $U(\p)$,
which is generated by declaring for each $r \geq 0$ that $x \in \g(r)$
is of Kazhdan degree
$r+1$.
The associated graded superalgebra $\gr U(\p)$ is supercommutative and
is naturally identified with the symmetric superalgebra $S(\p)$, viewed as a
positively
graded algebra via the analogously defined {\em Kazhdan grading}.
The Kazhdan filtration on
$U(\mathfrak{\p})$ induces a Kazhdan filtration on $W_\pi \subseteq U(\p)$
so that $\gr W_\pi\subseteq\gr U(\p) = S(\p)$.

\begin{Theorem}\label{T:WPBW}
Let $p:S(\p) \to S(\g^e)$
be the homomorphism induced by the projection of $\p$ onto $\g^e$
along (\ref{e:psum}).
The restriction of $p$
defines an isomorphism of Kazhdan-graded superalgebras
$\gr W_\pi
\stackrel{\sim}{\rightarrow}
S(\g^e).$
\end{Theorem}

\begin{proof}
Superize the arguments in \cite{GG}
as suggested in \cite[Remark 3.10]{Zhao}.
\end{proof}

In order to apply Theorem~\ref{T:WPBW}, it is helpful to have available an
explicit basis for the centralizer $\g^e$.
We say that a shift matrix $\sigma = (s_{i,j})_{1 \leq i,j \leq 2}$
is {\em compatible with $\pi$}
if either $k > 0$ and $\pi$ has $s_{2,1}$ columns of height one on its
left side
and $s_{1,2}$ columns of height one on its right side,
or if $k = 0$ and $l = s_{2,1}+s_{1,2}$.
These conditions determine a unique shift matrix $\sigma$
when $k > 0$, but there is some
minor ambiguity if $k = 0$ (which should never cause any concern).
For example if $\pi$ is as in
(\ref{piex})
then $\sigma = \left(\begin{array}{ll}0&2\\1&0\end{array}\right)$
is the only compatible shift matrix.

\begin{Lemma} \label{L:cent}
Let $\sigma = (s_{i,j})_{1 \leq i,j \leq 2}$
be a shift matrix compatible with $\pi$.
For $r \geq 0$,
let
$$
x_{i,j}^{(r)} \coloneqq \sum_{\substack{1 \le p,q \le m+n \\ \row(p) = i, \row(q) = j \\
\deg(e_{p,q}) = r-1}} \bar e_{p,q} \in \g(r-1).
$$
Then the elements
$\{x_{1,1}^{(r)}\:|\:0 < r \leq k\}\cup\{x_{2,2}^{(r)}\:|\:0 < r \leq l\}
\cup \{x_{1,2}^{(r)}\:|\:s_{1,2} < r \leq s_{1,2}+k\}
\cup \{x_{2,1}^{(r)}\:|\:s_{2,1} < r \leq s_{2,1}+k\}$
give a homogeneous basis for $\g^e$.
\end{Lemma}

\begin{proof}
As $e$ is even,
the centralizer of $e$ in $\g$ is just the same as a vector space as the
centralizer of $e$ viewed as an element of $\gl_{m+n}(\C)$,
so this follows as a special case of \cite[Lemma 7.3]{BK} (which
is \cite[IV.1.6]{SS}).
\end{proof}

We come to the key ingredient in
our approach: the explicit definition of special
elements of $U(\p)$ some of which turn out to generate $W_\pi$.
Define another ordering $\prec$ on the set $\{1,\dots,m+n\}$ by
declaring that
$i \prec j$ if $\col(i) < \col(j)$, or if $\col(i) = \col(j)$ and
$\row(i) < \row(j)$.
Let $\tilde\rho \in \mathfrak{t}^*$ be the weight with
\begin{equation}\label{rhodef}
(\tilde\rho|\eps_j) = \#\left\{\,i\:\big|\:i \preceq j\text{ and } |\row(i)| = \1\right\}
- \#\left\{\,i\:\big|\:i \prec j\text{ and } |\row(i)| = \0\right\}.
\end{equation}
For example if $\pi$ is as in (\ref{piex})
then $\tilde\rho = -\eps_4 - 2 \eps_5$.
The weight $\tilde\rho$ extends to a character of $\mathfrak{p}$,
so there are automorphisms
\begin{equation}
S_{\pm\tilde\rho}:U(\mathfrak{p}) \rightarrow U(\mathfrak{p}),
\qquad
e_{i,j} \mapsto e_{i,j} \pm \delta_{i,j} \tilde\rho(e_{i,i}).
\end{equation}
Finally, given $1 \leq i,j \leq 2, 0 \leq \varsigma \leq 2$ and
$r \geq 1$,
we define
\begin{equation} \label{EQ:Tpidef}
t_{i,j; \varsigma}^{(r)} \coloneqq
S_{\tilde\rho}\left(\sum_{s=1}^r (-1)^{r-s}
  \sum_{\substack{i_1,\dots,i_s \\ j_1,\dots,j_s}}
(-1)^{\#\left\{a=1,\dots,s-1\:|\:
\row(j_a) \leq\varsigma
\right\}}
\bar e_{i_1,j_1} \cdots \bar e_{i_s,j_s}\right),
\end{equation}
where the sum is over all $1 \le i_1, \dots , i_s, j_1, \dots , j_s
\le m+n$ such that
\begin{itemize}
\item $\row(i_1) = i$ and $\row(j_s) = j$;
\item $\col(i_a)\le\col(j_a)$ ($a=1,\dots,s$);
\item $\row(i_{a+1})=\row(j_{a})$ ($a=1,\dots,s-1$);
\item if $\row(j_a) > \varsigma$ then $\col(i_{a+1}) >\col(j_a)$
  ($a=1,\dots,s-1$);
\item if $\row(j_a) \leq \varsigma$ then $\col(i_{a+1}) \leq\col(j_a)$
  ($a=1,\dots,s-1$);
\item $\deg(e_{i_1,j_1}) + \dots + \deg(e_{i_s,j_s}) = r-s$.
\end{itemize}
It is convenient to collect these elements together into the
generating function
\begin{equation}
t_{i,j;\varsigma}(u) \coloneqq \sum_{r \geq 0}
t_{i,j;\varsigma}^{(r)} u^{-r} \in U(\p)\llbracket
u^{-1} \rrbracket,
\end{equation}
setting $t_{i,j;\varsigma}^{(0)} \coloneqq \delta_{i,j}$.
The following two
propositions should already convince the reader of the remarkable nature of
these elements.

\begin{Proposition}\label{factorixatin}
The following identities hold in $U(\p)\llbracket u^{-1} \rrbracket$:
\begin{align}
t_{1,1;1}(u) &= t_{1,1;0}(u)^{-1},\label{inv1}\\
t_{2,2;2}(u) &= t_{2,2;1}(u)^{-1},\label{inv2}\\
 t_{1,2;0}(u) &= t_{1,1;0}(u) t_{1,2;1}(u),\label{inv3}\\
t_{2,1;0}(u) &= t_{2,1;1}(u) t_{1,1;0}(u),\label{inv4}\\
t_{2,2;0}(u) &= t_{2,2;1}(u) + t_{2,1;1}(u) t_{1,1;0}(u)t_{1,2;1}(u).\label{inv5}
\end{align}
\end{Proposition}

\begin{proof}
This is proved in \cite[Lemma 9.2]{BK}; the argument there is entirely
formal and does not depend on the underlying associative algebra in
which the
calculations are performed.
\end{proof}

\begin{Proposition}\label{invariant}
Let $\sigma$ be a shift matrix compatible with $\pi$.
The following elements of $U(\p)$ belong to $W_\pi$: all
$t_{1,1;0}^{(r)}, t_{1,1;1}^{(r)},t_{2,2;1}^{(r)}$ and
$t_{2,2;2}^{(r)}$ for $r > 0$;
all $t_{1,2;1}^{(r)}$ for $r > s_{1,2}$; all
$t_{2,1;1}^{(r)}$ for $r > s_{2,1}.$
\end{Proposition}

\begin{proof}
Postponed to the next section.
\end{proof}

Now we can deduce our main result.
For any shift matrix $\sigma$ compatible with $\pi$, we identify
$U(\h)$
with the algebra $U_\sigma^l$ from (\ref{usl}) so that
$$
e_{i,j} \equiv
\left\{
\begin{array}{ll}
1^{\otimes (c-1)} \otimes e_{\row(i),\row(j)} \otimes 1^{\otimes (l-c)}
&\text{if $q_c = 2$,}\\
1^{\otimes (c-1)} \otimes e_{1,1} \otimes 1^{\otimes (l-c)}
&\text{if $q_c = 1$,}
\end{array}
\right.
$$
for any $1 \leq i,j \leq m+n$ with $c \coloneqq \col(i) =\col(j)$, where $q_c$ denotes the
number of boxes in this column of $\pi$.
Define the {\em Miura transform}
\begin{equation}\label{miura}
\mu : W_\pi \to U(\h) = U_\sigma^l
\end{equation}
to be the restriction to $W_\pi$
of the
shift automorphism $S_{-\tilde \rho}$ composed with
the natural homomorphism $\operatorname{pr}:U(\p) \rightarrow U(\h)$ induced by the
projection $\p \twoheadrightarrow \h$.

\begin{Theorem}\label{T:main}
Let $\sigma$ be a shift matrix compatible with $\pi$.
The Miura transform
is injective and its image is the algebra
$Y_\sigma^l \subseteq U_\sigma^l$ from (\ref{shy}).
Hence it defines a superalgebra isomorphism
\begin{equation}\label{mud}
\mu:W_\pi \stackrel{\sim}{\rightarrow} Y_\sigma^l
\end{equation}
between $W_\pi$ and the shifted Yangian of level $l$.
Moreover $\mu$
maps the invariants from Proposition~\ref{invariant} to the Drinfeld generators
of $Y^l_\sigma$ as follows:
\begin{align}\label{maini}
\mu(t_{1,1;0}^{(r)}) &= d_1^{(r)} \qquad(r > 0),&
\mu(t_{1,1;1}^{(r)}) &= \tilde d_1^{(r)}\qquad(r > 0),\\
\mu(t_{2,2;1}^{(r)}) &= d_2^{(r)}\qquad(r > 0),
&\mu(t_{2,2;2}^{(r)}) &= \tilde
d_2^{(r)}
\qquad(r > 0),\\
\mu(t_{1,2;1}^{(r)}) &= e^{(r)}\qquad (r > s_{1,2}),
&\mu(t_{2,1;1}^{(r)}) &= f^{(r)} \qquad (r > s_{2,1}).\label{mainj}
\end{align}
\end{Theorem}

\begin{proof}
We first establish the identities (\ref{maini})--(\ref{mainj}).
Note that the identities involving $\tilde d_i^{(r)}$
are consequences of the ones involving $d_i^{(r)}$
thanks to (\ref{inv1})--(\ref{inv2}),
recalling also that $\tilde d_i(u) = d_i(u)^{-1}$.
To prove all the other identities,
we proceed by induction on $s_{2,1}+s_{1,2} = l-k$.

First consider the base case $l=k$.
For $1 \leq i,j \leq 2$ and $r > 0$
we know in this situation
that $t_{i,j;0}^{(r)} \in W_\pi$ since,
using (\ref{inv3})--(\ref{inv5}), it can be expanded in terms of
elements all of which are known to
lie in $W_\pi$ by Proposition~\ref{invariant};
see also Lemma~\ref{invariant1} below.
Moreover, we have directly from (\ref{EQ:Tpidef}) and (\ref{new})
that $\mu(t_{i,j;0}^{(r)}) = t_{i,j}^{(r)} \in Y_\sigma^l$.
Hence $\mu(t_{i,j;0}(u)) = t_{i,j}(u)$.
The result follows from this,
(\ref{form1})--(\ref{form2}), and the analogous expressions for
$t_{1,1;0}(u), t_{2,2;1}(u),
t_{1,2;1}(u)$ and $t_{2,1;1}(u)$ derived from
(\ref{inv3})--(\ref{inv5}).

Now consider the induction step, so $s_{2,1}+s_{1,2} > 0$.
There are two cases according to whether $s_{2,1} > 0$ or $s_{1,2} >
0$.
We just explain the argument for the latter situation, since the
former is entirely similar.
Let $\dot\pi$ be the pyramid obtained from $\pi$ by removing the
rightmost column and let $W_{\dot\pi}$ be the corresponding finite
$W$-algebra.
We denote its Miura transform by $\dot\mu:W_{\dot\pi} \rightarrow
U_{\sigma_+}^{l-1}$, and similarly decorate all other notation
related to $\dot\pi$ with a dot to avoid confusion.
Now we proceed to
show that
$\mu(t_{1,2;1}^{(r)}) = e^{(r)}$ for each $r > s_{1,2}$.
By induction, we know that
$\dot\mu(\dot t_{1,2;1}^{(r)}) = \dot e^{(r)}$
for each $r \geq s_{1,2}$.
But then it follows from the explicit form of (\ref{EQ:Tpidef}),
together with
(\ref{dp}) and the definition of the evaluation homomorphism (\ref{Evalhom}),
that
\begin{align*}
\mu(t_{1,2;1}^{(r)}) &= \dot\mu\big(\dot t_{1,2;1}^{(r)}\big) \otimes 1
+ (-1)^{|2|}\dot\mu\big(\dot t_{1,2;1}^{(r-1)}\big) \otimes e_{1,1}\\
&= \dot
e^{(r)}
\otimes 1 + (-1)^{|2|} \dot e^{(r-1)} \otimes e_{1,1}
= e^{(r)},
\end{align*}
providing $r > s_{1,2}$.
The other cases are similar.

Now we deduce the rest of the theorem from
(\ref{maini})--(\ref{mainj}).
Order the elements of the set
\begin{multline*}
\Omega \coloneqq \{t_{1,1;0}^{(r)}\:|\:0 < r \leq k\}
\cup \{t_{2,2;1}^{(r)}\:|\:0 < r \leq l\}\\
\cup
\{t_{1,2;1}^{(r)}\:|\:s_{1,2}<r \leq s_{1,2}+k\}
\cup
\{t_{2,1;1}^{(r)}\:|\:s_{2,1}<r \leq s_{2,1}+k\}
\end{multline*}
in some way.
By Proposition~\ref{invariant},
each $t_{i,j;\varsigma}^{(r)} \in \Omega$
belongs to $W_\pi$.
Moreover,
from the definition (\ref{EQ:Tpidef}),
it is in filtered degree $r$ and
$\gr_r t_{i,j;\varsigma}^{(r)}$ is equal up to a sign to
the element $x_{i,j}^{(r)}$ from Lemma~\ref{L:cent}
plus a linear combination of monomials in
elements of strictly smaller Kazhdan degree.
Using Theorem~\ref{T:WPBW}, we deduce that the set of all ordered
supermonomials in the set $\Omega$ gives a linear basis for $W_\pi$.
By (\ref{maini})--(\ref{mainj}) and Corollary~\ref{pbwt},
$\mu$ maps this basis onto a basis for $Y_\sigma^l \subseteq
U_\sigma^l$.
Hence $\mu$ is an isomorphism.
\end{proof}

\begin{Remark}
The grading $\p = \bigoplus_{r \geq 0} \g(r)$
induces a grading on the superalgebra $U(\p)$.
However $W_\pi$ is not a graded subalgebra.
Instead, we get induced another filtration on $W_\pi$,
with respect to which the associated graded superalgebra $\gr' W_\pi$
is identified with
a graded subalgebra of $U(\p)$.
Each of the invariants
$t_{i,j;\varsigma}^{(r)}$ from Proposition~\ref{invariant}
belongs to filtered degree $(r-1)$ and
has image
$(-1)^{r-1} x_{i,j}^{(r)}$ in the associated graded algebra.
Combined with
Lemma~\ref{L:cent} and the usual PBW theorem for $\g^e$, it follows
that $\gr' W_\pi = U(\g^e)$.
Moreover
this filtration on $W_\pi$ corresponds
under the isomorphism $\mu$ to the filtration on $Y_\sigma^l$
induced by the Lie filtration on $Y_\sigma$.
\end{Remark}

\begin{Remark}
In this section, we have worked with
the ``right-handed'' definition (\ref{wpidef}) of the finite
$W$-algebra.
One can also consider the ``left-handed'' version
$$
W^{\dagger}_\pi \coloneqq
\{u \in U(\p) \:|\:  \m_\chi u \subseteq U(\g) \m_\chi \}.
$$
There is an analogue of Theorem \ref{T:main} for $W^\dagger_\pi$,
via which one sees that $W_\pi \cong W^\dagger_\pi$.
More precisely, we define the ``left-handed'' Miura transform $\mu^\dagger:W^\dagger_\pi
\rightarrow
U(\h)$ as above but twisting with the shift automorphism $S_{-\tilde\rho^\dagger}$ rather than
$S_{-\tilde\rho}$,
where
\begin{equation}
(\tilde\rho^\dagger|\eps_j) =
\#\left\{\,i\:\big|\:i \preceq^\dagger j\text{ and }
|\row(i)| = \1\right\}-
\#\left\{\,i\:\big|\:i \prec^\dagger j\text{ and }
|\row(i)| = \0\right\}
\end{equation}
and $i \prec^\dagger j$ means either $\col(i) > \col(j)$, or $\col(i) =
\col(j)$ and $\row(i) < \row(j)$.
The analogue of Theorem~\ref{T:main} asserts that $\mu^\dagger$
is injective with the same image as $\mu$.
Hence $\mu^{-1}\circ
\mu^\dagger$, i.e.\ the restriction of the shift
$S_{\tilde\rho-\tilde\rho^\dagger}:U(\p)\rightarrow U(\p)$,
gives an isomorphism between
$W^\dagger_\pi$ and $W_\pi$.
Noting that
\begin{equation}
\tilde\rho - \tilde\rho^\dagger =
\sum_{\substack{1 \leq i,j \leq m+n \\ \col(i) < \col(j)}} (-1)^{|\row(i)|+|\row(j)|} (\eps_i-\eps_j),
\end{equation}
there is a more conceptual explanation for this isomorphism
along the lines of the proof given in the
non-super case in \cite[Corollary 2.9]{BGK}.
\end{Remark}

\begin{Remark}
Another consequence of Theorem~\ref{T:main} together with
Remarks~\ref{iotasr}
and \ref{nodef}
is that up to isomorphism
the algebra $W_\pi$ depends only on the set $\{m,n\}$, i.e.\ on
the isomorphism type of $\mathfrak{g}$,
not on the particular
choice of the pyramid $\pi$ or the parity sequence.
As observed in \cite[Remark 3.10]{Zhao}, this can also be proved
by mimicking \cite[Theorem 2]{BG}.
\end{Remark}

\section{Proof of invariance} \label{S:proof}

In this section, we prove Proposition~\ref{invariant}.
We keep all notation as in the statement of the proposition.
Showing that $u \in U(\p)$ lies in the algebra
$W_\pi$ is equivalent to showing that
$[x,u] \in \m_\chi U(\g)$
for all $x \in \m$, or even just for all $x$
in a set of generators for $\m$.
Let
\begin{equation}\label{om}
\Omega \coloneqq
\{t_{1,1;0}^{(r)}\:|\:r > 0\}
\cup \{t_{1,2;1}^{(r)}\:|\:r > s_{1,2}\} \cup \{t_{2,1;1}^{(r)}\:|\:r
> s_{2,1}\}\cup\{t_{2,2;1}^{(r)}\:|\:r > 0\}.
\end{equation}
Our goal is to show that $[x,u] \in \m_\chi U(\g)$
for $x$ running over a set of generators of $\m$
and $u \in \Omega$. Proposition~\ref{invariant}
follows from this since all the other elements listed in the statement
of the proposition
can be expressed in terms of elements of $\Omega$ thanks to
Proposition~\ref{factorixatin}.
Also observe for the present purposes
that there is some freedom in the choice of the
weight $\tilde \rho$:
it can be adjusted by adding on any multiple of
``supertrace''
$\eps_1+\cdots+\eps_m-\eps_{m+1}-\cdots-\eps_{m+n}$. This just twists
the elements $t_{i,j;\varsigma}^{(r)}$ by an automorphism of $U(\g)$ so
does not have any effect on whether they belong to $W_\pi$.
So sometimes in this section we will allow ourselves to change the choice of
$\tilde\rho$.

\begin{Lemma}\label{invariant1}
Assuming $k=l$, we have that
$[x,t_{i,j;0}^{(r)}] \in \m_\chi U(\g)$ for all $x \in \m$ and
$r > 0$.
\end{Lemma}

\begin{proof}
Note when $k=l$
that $\tilde\rho = 0$ if $(|1|,|2|) = (\0,\1)$ and $\tilde\rho =
\eps_1+\cdots+\eps_m-\eps_{m+1}-\cdots-\eps_{m+n}$ if $(|1|,|2|) = (\1,\0)$.
As noted above, it does no harm to change the choice of $\tilde\rho$ to
assume in fact that $\tilde\rho = 0$ in both cases.
Now we proceed to mimic the argument in \cite[$\S$12]{BK}.

Consider the tensor algebra $T(M_l)$
in the (purely even)
vector space $M_l$ of $l\times l$ matrices over $\C$.
For $1 \leq i,j \leq 2$, define a linear map
$t_{i,j}:T(M_l) \rightarrow U(\mathfrak{g})$
by setting
\begin{align*}
t_{i,j}(1) \coloneqq \delta_{i,j},\qquad&\qquad
t_{i,j}(e_{a,b}) \coloneqq (-1)^{|i|} e_{i*a, j*b},\\
t_{i,j}(x_1\otimes\cdots\otimes x_r)
&\coloneqq \sum_{1 \leq h_1,\dots,h_{r-1} \leq 2}
t_{i,h_1}(x_1) t_{h_1,h_2}(x_2)\cdots t_{h_{r-1},j}(x_r),
\end{align*}
for $1 \leq a,b \leq p$, $r \geq 1$ and
$x_1,\dots,x_r \in M_l$,
where $i*a$ denotes $a$ if $|i| = \0$ and $l+a$ if $|i|=\1$.
It is straightforward to check for $x, y_1,\dots,y_r \in M_l$ that
\begin{multline}\label{iden}
[t_{i,j}(x), t_{p,q}(y_1\otimes\cdots\otimes y_r)]
=\\
(-1)^{|i||j|+|i||p|+|j||p|}
\sum_{s=1}^r
\big(
t_{p,j}(y_1\otimes\cdots\otimes y_{s-1})
t_{i,q}(x y_s \otimes\cdots\otimes y_r)\\
-
t_{p,j}(y_1\otimes\cdots\otimes y_{s} x)
t_{i,q}(y_{s+1} \otimes\cdots\otimes y_r)
\big),
\end{multline}
where the products $x y_s$ and $y_s x$
on the right are ordinary matrix products in $M_l$.
We extend $t_{i,j}$ to a $\C[u]$-module homomorphism
$T(M_l)[u] \rightarrow U(\g)[u]$ in the obvious way.
Introduce the following matrix with entries in the algebra
$T(M_l)[u]$:
$$
A(u) \coloneqq \left(
\begin{array}{cccccc}
u+e_{1,1} & e_{1,2} & e_{1,3}&\cdots & e_{1,l}\\
1&u+e_{2,2}&&&\vdots\\
0&&\ddots&&e_{l-2,l}\\
\vdots&&1&u+e_{l-1,l-1}&e_{l-1,l}\\
0&\cdots&0&1&u+e_{l,l}
\end{array}
\right)
$$
The point is that $t_{i,j;0}(u) = t_{i,j}(\cdet A(u))$,
where the {\em column determinant}
of an $l \times l$ matrix $A = (a_{i,j})$
with entries in a non-commutative ring means the Laplace
expansion
keeping all the monomials in column order,
i.e.\
$\cdet A  \coloneqq \sum_{w \in S_l} \operatorname{sgn}(w) a_{w(1),1}
\cdots a_{w(l), l}$.
We also write $A_{c, d}(u)$ for the submatrix of
$A(u)$ consisting only of rows and columns
numbered $c,\dots,d$.

Since $\mathfrak m$ is generated by elements of the form
$t_{i,j}(e_{c+1,c})$, it suffices now to show that
$[t_{i,j}(e_{c+1,c}),t_{p,q}(\cdet A(u))]) \in \m_\chi U(\g)$
for every $1 \leq i,j,p,q \leq 2$ and
$c = 1,\dots,l-1$.
To do this, we compute using the identity (\ref{iden}):
\begin{multline*}
[t_{i,j}(e_{c+1,c}), t_{p,q}(\cdet A(u))]
=\\t_{p,j}(\cdet A_{1, c-1}(u))
t_{i,q}
\left(
\cdet
\left(
\begin{array}{ccccc}
e_{c+1,c}&e_{c+1,c+1}&\cdots&e_{c+1,l}\\
1&u+e_{c+1,c+1} &\cdots & e_{c+1,l}\\\vdots&&\ddots&\vdots\\
0&\cdots&1&u+e_{l,l}
\end{array}
\right)\right)\\
-
t_{p,j}
\left(
\cdet
 \left(
\begin{array}{ccccc}
u+e_{1,1} &\cdots & e_{1,c}&e_{1,c}\\
1&\ddots&&\vdots\\
\vdots&&u+e_{c,c}&e_{c,c}\\
0&\cdots&1&e_{c+1,c}
\end{array}
\right)\right)t_{i,q}(\cdet A_{c+2, l}(u)).
\end{multline*}
In order to simplify the second term on right hand side, we observe
crucially for $h=1,2$ that
$t_{h,j}\left(\left(u+e_{c,c}\right) e_{c+1,c}\right)
\equiv
t_{h,j}\left(u+e_{c,c}\right) \pmod{\m_\chi U(\g)}.$
Hence,
we get that
\begin{multline*}
[t_{i,j}(e_{c+1,c}), t_{p,q}(\cdet A(u))]
\equiv\\
t_{p,j}(\cdet A_{1,c-1}(u))t_{i,q}
\left(
\cdet
\left(
\begin{array}{ccccc}
1&e_{c+1,c+1}&\cdots&e_{c+1,l}\\
1&u+e_{c+1,c+1} &\cdots & e_{c+1,l}\\\vdots&&\ddots&\vdots\\
0&\cdots&1&u+e_{l,l}
\end{array}
\right)\right)\\
-
t_{p,j}
\left(
\cdet
 \left(
\begin{array}{ccccc}
u+e_{1,1} &\cdots & e_{1,c}&e_{1,c}\\
1&\ddots&&\vdots\\
\vdots&&u+e_{c,c}&e_{c,c}\\
0&\cdots&1&1
\end{array}
\right)\right) t_{i,q} ( \cdet A_{c+2,l}(u))
\end{multline*}
modulo $\m_\chi U(\g)$.
Making the obvious row and column operations
gives that
\begin{align*}
\cdet\left(
\begin{array}{ccccc}
1&e_{c+1,c+1}&\cdots&e_{c+1,l}\\
1&u+e_{c+1,c+1} &\cdots & e_{c+1,l}\\\vdots&&\ddots&\vdots\\
0&\cdots&1&u+e_{l,l}
\end{array}
\right)
&
=
u \cdet A_{c+2,l}(u),\\
\cdet \left(
\begin{array}{ccccc}
u+e_{1,1} &\cdots & e_{1,c}&e_{1,c}\\
1&\ddots&&\vdots\\
\vdots&&u+e_{c,c}&e_{c,c}\\
0&\cdots&1&1
\end{array}
\right)
&=u \cdet A_{1,c-1}(u).
\end{align*}
It remains to substitute these into the preceeding formula.
\end{proof}

We are ready to prove Proposition~\ref{invariant}.
Our argument goes by induction on
$s_{2,1}+s_{1,2} = l-k$.
For the base case $k=l$,
we use Proposition~\ref{factorixatin} to rewrite the elements of $\Omega$
in terms of the elements $t_{i,j;0}^{(r)}$. The latter
lie in $W_\pi$ by Lemma~\ref{invariant1}. Hence so do the former.

Now assume that $s_{2,1}+s_{1,2} > 0$.
There are two cases according to to whether $s_{1,2} \geq s_{2,1}$
or $s_{2,1} > s_{1,2}$.
Suppose first that $s_{1,2} \geq s_{2,1}$, hence that $s_{1,2} > 0$.
We may as well assume in addition that $l \geq 2$:
the result is trivial for $l \leq 1$ as $\m = \{0\}$.
Let $\dot\pi$ be the pyramid obtained from $\pi$ by removing the
rightmost column.
We will decorate all notation related to $\dot\pi$ with a dot to avoid
any confusion. In particular, $W_{\dot\pi}$ is a subalgebra of
$U(\dot\p) \subseteq U(\dot\g)$.
Let
$$
\theta:U(\dot\g) \hookrightarrow U(\g)
$$
be the embedding sending $e_{i,j} \in \dot\g$ to $e_{i', j'}
\in \g$ if the $i$th and $j$th boxes of $\dot\pi$
correspond to the $i'$th and $j'$th boxes of $\pi$,
respectively.
Let $b$ be the label of the box at the end of the second row of
$\pi$, i.e.\ the box that gets
removed when passing from $\pi$ to $\dot\pi$. Also in the case that $s_{1,2} = 1$ let
 $c$ be the label of the box at the end
of the first row of $\pi$.

\begin{Lemma}\label{ind1}
In the above notation, the following hold:
\begin{itemize}
\item[(i)]
$t_{1,1;0}^{(r)} = \theta(\dot t_{1,1;0}^{(r)})$ for all $r > 0$;
\item[(ii)]
$t_{2,1;1}^{(r)} = \theta(\dot t_{2,1;1}^{(r)})$
for all $r > s_{2,1}$;
\item[(iii)]
$t_{1,2;1}^{(r)} =
\theta(\dot t_{1,2;1}^{(r)})
+ \theta(\dot t_{1,2;1}^{(r-1)}) S_{\tilde\rho}(\bar e_{b,b})
- \big[\theta(\dot t_{1,2;1}^{(r-1)}), e_{b-1,b}\big]$
for all $r > s_{1,2}$;
\item[(iv)]
$t_{2,2;1}^{(r)} = \theta(\dot t_{2,2;1}^{(r)})
+ \theta(\dot t_{2,2;1}^{(r-1)}) S_{\tilde\rho}(\bar e_{b,b})
- \big[\theta(\dot t_{2,2;1}^{(r-1)}), e_{b-1,b}\big]$
for all $r > 0$.
\end{itemize}
\end{Lemma}

\begin{proof}
This follows directly from the definition of these elements, using
also that
$\theta \circ S_{\dot{\tilde\rho}} = S_{\tilde\rho} \circ \theta$
on elements of $U(\dot\p)$.
\end{proof}

Observe next that $\m$ is generated by
$\theta(\dot\m) \cup J$ where
\begin{equation}
J \coloneqq
\left\{
\begin{array}{ll}
\{e_{b,c},e_{b,b-1}\}&\text{if $s_{1,2}=1$,}\\
\{e_{b,b-1}\}&\text{if $s_{1,2}>1$.}
\end{array}\right.
\end{equation}
We know by induction that
the following elements of $U(\dot\p)$
belong to $W_{\dot\pi}$:
all $\dot t_{1,1;0}^{(r)}$ and $\dot t_{2,2;1}^{(r)}$ for $r \geq 0$;
all $\dot t_{1,2;1}^{(r)}$ for $r \geq s_{1,2}$;
all $\dot t_{2,1;1}^{(r)}$ for $r > s_{2,1}$.
Also note that the elements of $\theta(\dot\m)$ commute
with $e_{b-1,b}$ and $S_{\tilde\rho}(\bar e_{b,b})$.
Combined with Lemma~\ref{ind1}, we deduce that
$[\theta(x),u] \in \theta(\dot\m_\chi) U(\g) \subseteq \m_\chi U(\g)$
for any $x \in \dot\m$ and
$u \in \Omega$.
It remains to show  that
$[x,u] \in \m_\chi U(\g)$
for each $x \in J$ and $u \in \Omega$.
This is done in Lemmas~\ref{l1}, \ref{l2} and \ref{l4} below.

\begin{Lemma}\label{l1}
For $x \in J$ and
$u \in \{t_{1,1;0}^{(r)}\:|\:r > 0\}\cup\{t_{2,1;1}^{(r)}\:|\:r >
s_{2,1}\}$, we have that
$[x,u] \in \m_\chi U(\g)$.
\end{Lemma}

\begin{proof}
Take $e_{b,d} \in J$.
Consider a monomial $S_{\tilde\rho}(\bar e_{i_1,j_1}\cdots \bar
e_{i_s,j_s})$
in the expansion of $u$ from (\ref{EQ:Tpidef}).
The only way it could fail to
supercommute with $e_{b,d}$ is if
it involves some $\bar e_{i_h,j_h}$ with
$j_h = b$ or $i_h = d$.
Since $\row(j_s) = 1$
and $\col(i_{h+1}) > \col(j_h)$ when $\row(j_h) = 2$,
this situation arises only if $s_{1,2}=1$, $i_h =d$ and
$j_h = c$.
Then the supercommutator
$[e_{b,d},\bar e_{i_h,j_h}]$
equals
$\pm e_{b,c}$. It remains to repeat this argument to see that
we can move the resulting
$e_{b,c} \in \m_\chi$ to the beginning.
\end{proof}

It is harder to deal with the remaining elements
$t_{1,2;1}^{(r)}$ and $t_{2,2;1}^{(r)}$ of $\Omega$.
We follow different approaches according to whether $s_{1,2} > 1$
or $s_{1,2}=1$.

\begin{Lemma}\label{l2}
Assume that $s_{1,2}> 1$.
We have that $[e_{b,b-1},u] \in \m_\chi U(\g)$
for all $u \in \{t_{1,2;1}^{(r)}\:|\:r > s_{1,2}\}\cup\{t_{2,2;1}^{(r)}\:|\:r >
0\}$.
\end{Lemma}

\begin{proof}
We just explain in detail for $u= t_{1,2;1}^{(r)}$; the other case
follows the same pattern.
Let $\ddot\pi$ be the pyramid obtained
from $\pi$ by removing its rightmost two columns.
We decorate all notation associated to $W_{\ddot \pi}$
with a double dot, so $W_{\ddot\pi} \subseteq U(\ddot \p)
\subseteq U(\ddot \g)$ and so on.
Let
$$
\phi:U(\ddot\g) \hookrightarrow U(\g)
$$
be the embedding sending $e_{i,j} \in \ddot\g$ to $e_{i',j'} \in \g$
where the $i$th and $j$th boxes of $\ddot\pi$ are labelled by $i$ and
$j$ in $\pi$,
respectively.
For $r \geq s_{1,2}$, we have by analogy with Lemma~\ref{ind1}(iii) that
$$
\theta(\dot t_{1,2;1}^{(r)}) =
\phi(\ddot t_{1,2;1}^{(r)})
+ \phi(\ddot t_{1,2;1}^{(r-1)}) S_{\tilde\rho}(\bar e_{b-1,b-1})
- \big[\phi(\ddot t_{1,2;1}^{(r-1)}), e_{b-2,b-1}\big].
$$
We combine this with Lemma~\ref{ind1}(iii)
to deduce for $r > s_{1,2}$ that
\begin{align*}
t_{1,2;1}^{(r)}
=\,&\phi(\ddot t_{1,2;1}^{(r)})
+ \phi(\ddot t_{1,2;1}^{(r-1)}) S_{\tilde\rho}(\bar e_{b-1,b-1})
- \big[\phi(\ddot t_{1,2;1}^{(r-1)}),e_{b-2,b-1}\big]\\
&+\phi(\ddot t_{1,2;1}^{(r-1)}) S_{\tilde\rho}(\bar e_{b,b})
+ \phi(\ddot t_{1,2;1}^{(r-2)}) S_{\tilde\rho}(\bar
e_{b-1,b-1})S_{\tilde\rho}(\bar e_{b,b})
\\&
- \big[\phi(\ddot
t_{1,2;1}^{(r-2)}),e_{b-2,b-1}\big]S_{\tilde\rho}(\bar e_{b,b})-
\phi(\ddot t_{1,2;1}^{(r-2)})\bar e_{b-1,b}
+ \big[\phi(\ddot t_{1,2;1}^{(r-2)}),e_{b-2,b}\big].
\end{align*}
We deduce that
\begin{align*}
[e_{b,b-1},t_{1,2;1}^{(r)}]
=\,&
\phi(\ddot t_{1,2;1}^{(r-2)})
\big(\bar e_{b,b-1} S_{\tilde\rho}(\bar e_{b,b})
-
\bar e_{b,b-1}  S_{\tilde\rho}(\bar e_{b-1,b-1})
+(-1)^{|2|}\bar e_{b,b-1}
\big)\\
&
+ \big[\phi(\ddot
t_{1,2;1}^{(r-2)}),e_{b-2,b-1}\big]\bar e_{b,b-1}-
\phi(\ddot t_{1,2;1}^{(r-2)})(\bar e_{b,b}-\bar e_{b-1,b-1})\\
&- \big[\phi(\ddot t_{1,2;1}^{(r-2)}),e_{b-2,b-1}\big].
\end{align*}
Working modulo $\m_\chi U(\g)$, we can replace all
$\bar e_{b,b-1}$ by $1$. Then we are reduced just to checking that
$$
S_{\tilde\rho}(\bar e_{b,b})
- S_{\tilde\rho}(\bar e_{b-1,b-1})
+ (-1)^{|2|}
=
\bar e_{b,b}-\bar e_{b-1,b-1}.
$$
This follows because $(\tilde\rho|\eps_b)-(\tilde\rho|\eps_{b-1}) +
(-1)^{|2|} = 0$ by the definition (\ref{rhodef}).
\end{proof}

\begin{Lemma}\label{bird}
Assume that $s_{1,2}=1$.
For $r > 2$ we have that
\begin{align}\label{lastone}
t_{1,2;1}^{(r)}
&=
(-1)^{|1|}\big[t_{1,1;0}^{(2)}, t_{1,2;1}^{(r-1)}\big] -
t_{1,1;0}^{(1)} t_{1,2;1}^{(r-1)},\\
t_{2,2;1}^{(r)}
&=
(-1)^{|1|}\big[t_{1,2;1}^{(2)}, t_{2,1;1}^{(r-1)}\big] -
\sum_{a=0}^{r} t_{1,1;1}^{(a)} t_{2,2;1}^{(r-a)}.
\label{lasttwo}
\end{align}
\end{Lemma}

\begin{proof}
We prove (\ref{lastone}).
The induction
hypothesis means that we can appeal to Theorem~\ref{T:main}
for the algebra $W_{\dot\pi}$.
Hence using the relations from Theorem~\ref{T:gens},
we know that the following hold in the algebra
$W_{\dot\pi}$ for all $r \geq 2$:
\begin{align*}
\dot t_{1,2;1}^{(r)} = (-1)^{|1|} \big[ \dot t_{1,1;0}^{(2)}, \dot
t_{1,2;1}^{(r-1)} \big]
- \dot t_{1,1;0}^{(1)} \dot t_{1,2;1}^{(r-1)}.
\end{align*}
Using Lemma~\ref{ind1}, we deduce for $r > 2$
that
\begin{align*}
t_{1,2;1}^{(r)} =\,
&\theta(\dot t_{1,2;1}^{(r)})
+ \theta(\dot t_{1,2;1}^{(r-1)}) S_{\tilde\rho}(\bar e_{b,b})
- \big[\theta(\dot t_{1,2;1}^{(r-1)}), e_{b-1,b}\big]
\\
=\,&
(-1)^{|1|} \big[ t_{1,1;0}^{(2)}, \theta(\dot
t_{1,2;1}^{(r-1)}) \big]
- t_{1,1;0}^{(1)} \theta(\dot t_{1,2;1}^{(r-1)})\\
&+
(-1)^{|1|} \big[ t_{1,1;0}^{(2)}, \theta(\dot
t_{1,2;1}^{(r-2)})\big] S_{\tilde\rho}(\bar e_{b,b})
- t_{1,1;0}^{(1)} \theta(\dot t_{1,2;1}^{(r-2)})S_{\tilde\rho}(\bar e_{b,b})\\
&-(-1)^{|1|}
\Big[\big[ t_{1,1;0}^{(2)}, \theta(\dot
t_{1,2;1}^{(r-2)}) \big], e_{b-1,b}\Big]
+ \big[t_{1,1;0}^{(1)} \theta(\dot t_{1,2;1}^{(r-2)}),
e_{b-1,b}\big]\\
=\,&
(-1)^{|1|} \Big[ t_{1,1;0}^{(2)}, \theta(\dot
t_{1,2;1}^{(r-1)}) +\theta(\dot
t_{1,2;1}^{(r-2)}) S_{\tilde\rho}(\bar e_{b,b})
-
\big[\theta(\dot
t_{1,2;1}^{(r-2)}), e_{b-1,b}
\big]\Big]\\
&
- t_{1,1;0}^{(1)} \left(
\theta(\dot t_{1,2;1}^{(r-1)})
+ \theta(\dot t_{1,2;1}^{(r-2)})S_{\tilde\rho}(\bar e_{b,b})
- \big[\theta(\dot t_{1,2;1}^{(r-2)}),
e_{b-1,b}\big]\right)\\
=\,&
(-1)^{|1|} \big[t_{1,1;0}^{(2)}, t_{1,2;1}^{(r-1)}\big] - t_{1,1;0}^{(1)}t_{1,2;1}^{(r-1)}.
\end{align*}
The other equation (\ref{lasttwo}) follows by a similar
trick.
\end{proof}

\begin{Lemma}\label{l4}
Assume that $s_{1,2}=1$.
We have that $[x,u] \in \m_\chi U(\g)$
for all $x \in J$ and $u \in \{t_{1,2;1}^{(r)}\:|\:r > s_{1,2}\}
\cup \{t_{2,2;1}^{(r)}\:|\:r > 0\}$.
\end{Lemma}

\begin{proof}
Proceed by induction on $r$.
The base cases when $r \leq 2$ are small enough that they can be checked
directly from the definitions.
Then for $r > 2$ use Lemma~\ref{bird}, noting by the induction hypothesis and
Lemma~\ref{l1} that all the terms on the right hand side of
(\ref{lastone})--(\ref{lasttwo}) are already known to
lie in $\m_\chi U(\g)$.
\end{proof}

We have now verified the induction step in the case that
$s_{1,2} \geq s_{2,1}$. It remains to establish the induction step when
$s_{2,1} > s_{1,2}$.
The strategy for this is sufficiently similar to
case just done (based on removing columns from the left of the
pyramid $\pi$) that we leave
the details to the reader.
We just note one minor difference: in the proof of the analogue of Lemma~\ref{ind1}
it is no longer the case that
$\theta \circ S_{\dot{\tilde\rho}} = S_{\tilde\rho}\circ \theta$, but
this can be fixed by allowing the choice of $\tilde\rho$
to change by a multiple of $\eps_1+\cdots+\eps_m-\eps_{m+1}-\cdots-\eps_{m+n}$.

This completes the proof of Proposition~\ref{invariant}.

\section{Triangular decomposition}

Let $W_\pi$ be the principal $W$-algebra
in $\mathfrak{g} = \mathfrak{gl}_{m|n}(\C)$
associated to
pyramid $\pi$.
We adopt all the notation
from $\S$\ref{spyr}.
So:
\begin{itemize}
\item
$(|1|,|2|)$ is a parity sequence chosen so that
$(|1|,|2|) = (\0,\1)$ if $m < n$
and $(|1|,|2|) = (\1,\0)$ if $m > n$;
\item
$\pi$ has $k = \min(m,n)$ boxes in its first row and $l =
\max(m,n)$
boxes in its second row;
\item
$\sigma = (s_{i,j})_{1 \leq i,j \leq 2}$ is a shift matrix compatible with
$\pi$.
\end{itemize}
We identify $W_\pi$ with $Y_\sigma^{l}$,
the shifted Yangian of level $l$,
via the isomorphism $\mu$ from (\ref{mud}).
Thus we have available a set of
Drinfeld generators for $W_\pi$
satisfying the relations from Theorem~\ref{T:gens}
plus the additional truncation relations
$d_1^{(r)} = 0$ for $r > k$.
In view of (\ref{maini})--(\ref{mainj}) and (\ref{EQ:Tpidef}),
we even have available explicit formulae for these generators
as elements of $U(\p)$, although we seldom need to use these (but see
the proof of Lemma~\ref{ishw} below).

By the relations, $W_\pi$ admits a
$\Z$-grading
$$
W_\pi = \bigoplus_{g \in \Z} W_{\pi;g}
$$
such that the generators $d_i^{(r)}$ are of degree $0$,
the generators $e^{(r)}$ are of degree $1$, and the generators
$f^{(r)}$ are of degree $-1$.
Moreover the PBW theorem (Corollary~\ref{pbwt})
implies that $W_{\pi;g} = 0$ for $|g|>k$.

More surprisingly, the algebra $W_{\pi}$ admits a
triangular decomposition.
To introduce this,
let $W_{\pi}^0$ (resp.\ $W_{\pi}^+$, resp.\ $W_{\pi}^-$)
be the subalgebra of $W_{\pi}$ generated by the elements
$\Omega_0 \coloneqq \{d_1^{(r)}, d_2^{(s)}\:|\:0 < r \leq k, 0 < s \leq l\}$
(resp.\ $\Omega_+ \coloneqq \{e^{(r)}\:|\:s_{1,2} < r \leq s_{1,2}+k\}$,
resp.\ $\Omega_- \coloneqq \{f^{(r)}\:|\:s_{2,1} < r \leq s_{2,1}+k\}$).
Let $W_{\pi}^\sharp$ (resp.\ $W_{\pi}^{\flat}$)
be the subalgebra of $W_{\pi}$ generated by $\Omega_0 \cup \Omega_+$
(resp.\ by $\Omega_- \cup \Omega_0$).
We warn the reader that the elements
$e^{(r)}\:(r > s_{1,2}+k)$
do not necessarily lie in $W_{\pi}^+$
(but they do
lie in $W_{\pi}^\sharp$ by (\ref{van1})).
Similarly the elements $f^{(r)}\:(r > s_{2,1}+k)$
do not necessarily lie in $W_{\pi}^-$
(but they do lie in $W_{\pi}^\flat$),
and the elements
$d_2^{(r)}\:(r > l)$ do not necessarily lie in any of
$W_{\pi}^0, W_{\pi}^\sharp$ or $W_{\pi}^\flat$.

\begin{Theorem}\label{triangular}
The algebras $W_{\pi}^0, W_{\pi}^+$ and $W_{\pi}^-$
are free supercommutative superalgebras
on
generators
$\Omega_0$, $\Omega_+$ and $\Omega_-$, respectively.
Multiplication defines vector space isomorphisms
\begin{align*}
W_{\pi}^- \otimes W_{\pi}^0
\otimes &W_{\pi}^+ \stackrel{\sim}{\rightarrow} W_{\pi},\\
W_{\pi}^0
\otimes W_{\pi}^+ \stackrel{\sim}{\rightarrow} W_{\pi}^\sharp,\qquad
 &W_{\pi}^- \otimes W_{\pi}^0
\stackrel{\sim}{\rightarrow} W_{\pi}^\flat.
\end{align*}
Moreover, there are unique surjective homomorphisms
$$
W_{\pi}^\sharp \twoheadrightarrow W_{\pi}^0,
\qquad
W_{\pi}^\flat \twoheadrightarrow W_{\pi}^0
$$
sending
$e^{(r)} \mapsto 0$ for all $r > s_{1,2}$ or
$f^{(r)} \mapsto 0$ for all $r > s_{2,1}$, respectively,
such that the restriction of these maps
to the subalgebra $W_{\pi}^0$ is the identity.
\end{Theorem}

\begin{proof}
Throughout the proof,
we repeatedly apply the PBW theorem (Corollary~\ref{pbwt}),
choosing the order of generators so that $\Omega_- <
\Omega_0 < \Omega_+$.

To start with, note by the left hand relations in Theorem~\ref{T:gens}
that each of $W_{\pi}^0, W_{\pi}^+$ and $W_{\pi}^-$ is
supercommutative.
Combined with the PBW theorem, we deduce that they are free
supercommutative on the given generators.
Moreover the PBW theorem implies that the
multiplication map
$W_{\pi}^- \otimes W_{\pi}^0 \otimes W_{\pi}^+
\rightarrow W_{\pi}$ is a vector space isomorphism.

Next we observe that $W_{\pi}^\sharp$ contains all the elements
$e^{(r)}\:(r > s_{1,2})$. This follows from (\ref{van1}) by induction
on $r$.
Moreover it is spanned as a vector space by
the ordered supermonomials in the generators
$\Omega_0 \cup \Omega_+$.
This follows from (\ref{van1}), the relation for $[d_i^{(r)}, e^{(s)}]$ in
Theorem~\ref{T:gens}, and induction on Kazhdan degree.
Hence the multiplication map
$W_{\pi}^0 \otimes W_{\pi}^+ \rightarrow W_{\pi}^\sharp$ is
surjective.
It is injective by the PBW theorem, so it is an isomorphism.
Similarly $W_{\pi}^- \otimes W_{\pi}^0 \rightarrow W_{\pi}^\flat$
is an isomorphism.

Finally, let $J^\sharp$ be the two-sided ideal of $W_{\pi}^\sharp$
that is the
sum of all of the graded components $W_{\pi;g}^\sharp \coloneqq W_{\pi}^\sharp
\cap W_{\pi;g}$
for $g > 0$.
By the PBW theorem, the natural quotient map
$W_{\pi}^0 \rightarrow
W_{\pi}^\sharp / J^\sharp$ is an isomorphism.
Hence there is a surjection $W_{\pi}^\sharp \twoheadrightarrow
W_{\pi}^0$ as in the statement of the theorem.
A similar argument
yields the desired surjection
$W_{\pi}^\flat \twoheadrightarrow
W_{\pi}^0$.
\end{proof}

\section{Irreducible representations}

Continue with the notation of the previous section.
Using the triangular decomposition, we can classify irreducible
$W_{\pi}$-modules by highest weight theory.
Define a {\em $\pi$-tableau} to be a filling of the boxes of
the pyramid $\pi$ by
arbitrary complex numbers.
Let $\Tab_\pi$ denote the set of all such $\pi$-tableaux.
We represent the $\pi$-tableau with entries $a_1,\dots,a_k$ along its first
row and $b_1,\dots,b_{l}$ along its second row simply by the array
$\substack{a_1 \cdots a_{k} \\ b_1 \cdots b_{l}}$.
We say that $A, B\in \Tab_\pi$ are {\em row equivalent},
denoted $A \sim B$, if $B$ can be obtained from $A$ by permuting
entries within each row.

Recall from  Theorem~\ref{triangular} that
$W_\pi^0$ is the
polynomial algebra on $\{d_1^{(r)},d_2^{(s)}\:|\:0 < r \leq k, 0 < s \leq l\}$.
For
$A = \substack{a_1 \cdots a_{k} \\ b_1 \cdots b_{l}} \in \Tab_\pi$, let $\C_A$ be the one-dimensional
$W_{\pi}^0$-module on basis $1_A$
such that
\begin{align}\label{type1}
u^{k} d_1(u) 1_A &= (u+a_1) \cdots (u+a_k) 1_A,\\
u^{l} d_2(u) 1_A &= (u+b_1) \cdots (u+b_l) 1_A.\label{type2}
\end{align}
Thus $d_1^{(r)} 1_A = e_r(a_1,\dots,a_k) 1_A$
and $d_2^{(r)} 1_A = e_r(b_1,\dots,b_l) 1_A$,
where $e_r$ denotes the $r$th elementary symmetric polynomial.
Every irreducible $W_{\pi}^0$-module is isomorphic
to $\C_A$ for some $A \in \Tab_\pi$, and $\C_A \cong \C_B$ if and only if $A \sim B$.

Given $A \in \Tab_\pi$, we view $\C_A$ as a
$W_{\pi}^\sharp$-module via the surjection
$W_{\pi}^\sharp \twoheadrightarrow W_{\pi}^0$ from
Theorem~\ref{triangular}, i.e.\ $e^{(r)} 1_A = 0$ for all $r > s_{1,2}$.
Then we induce to form the {\em Verma module}
\begin{equation} \label{e:WVerma}
\overline{M}(A) \coloneqq W_{\pi} \otimes_{W_{\pi}^\sharp} \C_A.
\end{equation}
Sometimes we need to view this as a supermodule, which we do by
declaring that its cyclic generator $1 \otimes
1_A$ is even.
By Theorem~\ref{triangular},
$W_{\pi}$ is a free right $W_{\pi}^\sharp$-module with basis given by
the
ordered supermonomials in the odd elements $\{f^{(r)}\:|\:s_{2,1}<r \leq s_{2,1}+k\}$.
Hence $\overline{M}(A)$ has basis given by the vectors $x \otimes 1_A$
as $x$ runs over this set of supermonomials.
In particular $\dim \overline{M}(A) = 2^k$.

The following lemma shows that $\overline{M}(A)$ has a unique
irreducible quotient
which we denote by $\overline{L}(A)$;
we write $v_+$ for the image of $1 \otimes 1_A \in \overline{M}(A)$
in $\overline{L}(A)$.

\begin{Lemma}\label{irrdesc}
For
$A = \substack{a_1 \cdots a_{k} \\ b_1 \cdots b_{l}} \in \Tab_\pi$,
the Verma module $\overline{M}(A)$ has a unique
irreducible quotient $\overline{L}(A)$.
The image $v_+$ of $1 \otimes 1_A$ is the unique (up to scalars)
non-zero vector in $\overline{L}(A)$ such that
$e^{(r)} v_+ = 0$ for all $r > s_{1,2}$.
Moreover we have that
$
d_1^{(r)} v_+ = e_r(a_1,\dots,a_k) v_+$ and
$d_2^{(r)} v_+ = e_r(b_1,\dots,b_l) v_+$
for all $r \geq 0$.
\end{Lemma}

\begin{proof}
Let $\lambda \coloneqq (-1)^{|1|}(a_1+\cdots+a_k)$.
For any $\mu \in \C$,
let $\overline{M}(A)_\mu$ be the $\mu$-eigenspace
of the endomorphism of $\overline{M}(A)$
defined by $d \coloneqq (-1)^{|1|} d_1^{(1)} \in W_\pi$. Note
by (\ref{type1}) and the relations that
$d 1_A = \lambda 1_A$ and $[d,f^{(r)}] =
-f^{(r)}$ for each $r > s_{2,1}$.
Using the PBW basis for $\overline{M}(A)$, it follows that
\begin{equation}\label{espace}
\overline{M}(A) =
\bigoplus_{i=0}^k \overline{M}(A)_{\lambda-i}
\end{equation}
and $\dim \overline{M}(A)_{\lambda-i} = \binom{k}{i}$ for each $0 \leq
i \leq k$.
In particular, $\overline{M}(A)_\lambda$ is one-dimensional, and it
generates $\overline{M}(A)$ as a $W_\pi^\flat$-module.
This is all that is needed to deduce that $\overline{M}(A)$ has a unique irreducible quotient
$\overline{L}(A)$ following
the standard argument of highest weight theory.

The vector $v_+$ is a non-zero vector
annihilated by $e^{(r)}\:(r > s_{1,2})$, and
$d_1^{(r)} v_+$ and $d_2^{(r)} v_+$ are as stated
thanks to (\ref{type1})--(\ref{type2}).
It just remains to show that any vector $v \in \overline{L}(A)$ annihilated by
all
$e^{(r)}$ is a multiple of $v_+$.
The decomposition (\ref{espace}) induces an analogous decomposition
\begin{equation}
\overline{L}(A) = \bigoplus_{i=0}^k \overline{L}(A)_{\lambda-i},
\end{equation}
although for $0 < i \le k$
the eigenspace $\overline{L}(A)_{\lambda-i}$
may now be zero.
Write $v = \sum_{i=0}^k v_i$ with
$v_i \in \overline{L}(\lambda)_{\lambda-i}$.
Then we need to show that $v_i = 0$ for $i > 0$.
We have that $e^{(r)} v = \sum_{i=1}^k e^{(r)} v_i = 0$,
hence $e^{(r)} v_i = 0$ for each $i$.
But this means for $i > 0$ that the submodule $W_\pi v_i =
W_\pi^\flat v_i$
has trivial intersection with $\overline{L}(\lambda)_\lambda$,
hence it must be zero.
\end{proof}

Here is the classification of irreducible $W_\pi$-modules.

\begin{Theorem}\label{irrclass}
Every irreducible $W_{\pi}$-module is finite dimensional and is
isomorphic to one of the modules
$\overline{L}(A)$ from
Lemma~\ref{irrdesc} for some $A \in \Tab_\pi$.
Moreover $\overline{L}(A) \cong \overline{L}(B)$ if and only if $A
\sim B$.
Hence, fixing a set $\Tab_\pi / {\scriptstyle\sim}$ of representatives for the
$\sim$-equivalence classes in $\Tab_\pi$,
the modules
$$
\{\overline{L}(A)\:|\:A \in \Tab_\pi / {\scriptstyle\sim}\}
$$
give a complete set of pairwise inequivalent irreducible $W_\pi$-modules.
\end{Theorem}

\begin{proof}
We note to start with for $A, B \in \Tab_\pi$
that $\overline{L}(A) \cong \overline{L}(B)$ if and only if $A \sim
B$. This is clear from Lemma~\ref{irrdesc}.

Now take an arbitrary (conceivably infinite dimensional)
irreducible $W_{\pi}$-module $L$.
We want to show that $L \cong \overline{L}(A)$ for some $A \in \Tab_\pi$.
For $i \geq 0$, let
$$
L[i] \coloneqq
\left\{v \in L \:\big|\:
W_{\pi;g} v = \{0\} \text{ if }g > 0\text{ or }g \leq -i\right\}.
$$
We claim initially that $L[k+1] \neq \{0\}$.
To see this, recall that $W_{\pi;g} = \{0\}$ for $g \leq -k-1$,
so by the PBW theorem $L[k+1]$ is simply the set of all vectors $v \in L$
such that $e^{(r)} v = 0$ for all $s_{1,2} < r \leq s_{1,2}+k$.
Now take any non-zero vector $v \in L$ such that
$\#\{r=s_{1,2}+1,\dots,s_{1,2}+k\:|\:e^{(r)} v = 0\}$ is maximal.
If $e^{(r)} v \neq 0$ for some $s_{1,2}<r\leq s_{1,2}+k$, we can replace $v$ by
$e^{(r)} v$ to get a non-zero vector annihilated by more $e^{(r)}$'s.
Hence $v \in L[k+1]$ by the maximality of the choice of $v$,
and we have shown that $L[k+1] \neq \{0\}$.

Since $L[k+1] \neq \{0\}$ it makes sense to define $i \geq 0$ to be
minimal such that $L[i] \neq \{0\}$.
Since $L[0] = \{0\}$, we actually have that $i > 0$.
Pick $0 \neq v \in L[i]$ and let $L' \coloneqq W_{\pi}^\sharp v$. Actually, by
the PBW theorem, we have that $L' = W_{\pi}^0 v$, and
$L' \subseteq L[i]$.
Suppose first that $L'$ is irreducible as a $W_{\pi}^0$-module.
Then $L' \cong \C_A$ for some $A \in \Tab_\pi$.
The inclusion $L' \hookrightarrow L$ induces a non-zero
$W_{\pi}$-module homomorphism
$$
\overline{M}(A) \cong W_{\pi} \otimes_{W_{\pi}^\sharp} L'
\rightarrow L,
$$
which is surjective as $L$ is irreducible.
Hence $L \cong \overline{L}(A)$.

It remains to rule out the possibility that $L'$ is reducible.
Suppose for a contradiction that $L'$ possesses a non-zero
proper $W_{\pi}^0$-submodule $L''$.
As  $L = W_{\pi} L''$
and $W_{\pi}^\sharp L'' = L''$,
the PBW theorem implies that we can write
$$
v =
w+\sum_{h = 1}^k
\sum_{s_{2,1} < r_1 < \cdots < r_{h} \leq s_{2,1}+k}
f^{(r_1)}
\cdots f^{(r_{h})}
v_{r_1,\dots,r_{h}}
$$
for some vectors $v_{r_1,\dots,r_{h}}, w \in L''$.
Then we have that
$$
0 \neq v - w
\in L[i]
\cap \left(\sum_{g \leq -1} W_{\pi;g} L[i]\right) \subseteq L[i-1].
$$
This shows $L[i-1] \neq \{0\}$, contradicting the minimality of the choice of $i$.
\end{proof}

The final theorem of the section gives an explicit monomial basis for
$\overline{L}(A)$.
We only prove linear independence here; the
spanning part of the argument will be given in the next section.

\begin{Theorem}\label{moved}
Suppose
$A = \substack{a_1 \cdots a_{k} \\ b_1 \cdots b_{l}} \in \Tab_\pi$.
Let
$h \geq 0$ be maximal such that there exist distinct
$1 \leq i_1,\dots,i_h \leq k$ and distinct
$1 \leq j_1,\dots,j_h \leq l$ with $a_{i_1}=b_{j_1},\dots,a_{i_h} =
b_{j_h}$.
Then the irreducible module
$\overline{L}(A)$ has basis given by the vectors
$x v_+$ as $x$ runs over all ordered supermonomials in the odd
elements
$\{f^{(r)}\:|\:s_{2,1} < r \leq s_{2,1}+k-h\}$.
\end{Theorem}

\begin{proof}
Let $\bar k \coloneqq k-h$ and $\bar l \coloneqq l-h$.
Since $\overline{L}(A)$ only depends on the $\sim$-equivalence class
of $A$, we can reindex to
assume that $a_{\bar k+1} = b_{\bar l+1}, a_{\bar k+2}=b_{\bar
  l+2},\dots,a_{k} = b_{l}$.
We proceed to show that the vectors $x v_+$
for all
ordered supermonomials
$x$ in $\{f^{(r)}\:|\:s_{2,1}< r \leq s_{2,1}+\bar k\}$
are linearly independent in $\overline{L}(A)$.
In fact it is enough for this to show just that
\begin{equation}\label{goal}
 f^{(s_{2,1}+1)} f^{(s_{2,1}+2)}\cdots f^{(s_{2,1}+\bar k)} v_+ \neq 0.
\end{equation}
Indeed, assuming (\ref{goal}), we can prove the linear independence in general
by taking any non-trivial linear
relation of the form
$$
\sum_{a = 0}^{\bar k}
\sum_{s_{2,1} < r_1 < \cdots < r_a \leq s_{2,1}+\bar k}
\lambda_{r_1,\dots,r_a}
f^{(r_1)} \cdots f^{(r_a)} v_+ = 0.
$$
Let $a$ be minimal such that $\lambda_{r_1,\dots,r_a} \neq 0$ for some
$r_1,\dots,r_a$.
Apply $f^{(s_1)}\cdots f^{(s_{\bar k-a})}$
where $s_{2,1} < s_1 < \cdots < s_{\bar k-a} \leq s_{2,1}+\bar k$ are
different from $r_1 < \cdots < r_a$.
All but one term of the summation becomes zero and
using (\ref{goal}) we can deduce that $\lambda_{r_1,\dots,r_a} = 0$, a
contradiction.

In this paragraph, we prove (\ref{goal})
by showing that
\begin{equation}\label{goal2}
e^{(s_{1,2}+1)} e^{(s_{1,2}+2)} \cdots e^{(s_{1,2}+\bar k)}
 f^{(s_{2,1}+1)} f^{(s_{2,1}+2)}\cdots f^{(s_{2,1}+\bar k)} v_+ \neq 0.
\end{equation}
The left hand side of (\ref{goal2}) equals
$$
\sum_{w \in S_{\bar k}} \operatorname{sgn}(w)
\left[e^{(\bar k+1 +s_{1,2} - 1)},  f^{(s_{2,1}+w(1))}\right]
\cdots
\left[e^{(\bar k+1 +s_{1,2}- \bar k)},  f^{(s_{2,1}+w(\bar k))}\right]
v_+.
$$
By Remark~\ref{center}, up to a sign,
this is
$\det
\left(\tilde c^{(\bar l-i+j)} \right)_{1 \leq i,j \leq \bar k} v_+$.
It is easy to see from (i) that
$\tilde c^{(r)} v_+ = e_r(b_1,\dots,b_{\bar l} /
a_1,\dots, a_{\bar k}) v_+$
where
$$
e_r(b_1,\dots,b_{\bar l} / a_1,\dots,a_{\bar k})
\coloneqq
\sum_{s+t=r} (-1)^t e_s(b_1,\dots,b_{\bar l}) h_t(a_1,\dots,a_{\bar
  k})
$$
is the $r$th elementary supersymmetric function
from \cite[Exercise I.3.23]{Mac}.
Thus we need to show that
$\det \left(e_{\bar l - i + j}(b_1,\dots,b_{\bar l} /
  a_1,\dots,a_{\bar k}) \right)_{1 \leq i,j \leq \bar k} \neq 0$.
But this determinant is exactly the supersymmetric Schur function
$s_\lambda(b_1,\dots,b_{\bar l} / a_1,\dots,a_{\bar k})$
defined in \cite[Exercise I.3.23]{Mac} for the partition $\lambda = ({\bar
k}^{\bar l})$.
Hence by the factorization property described there,
it is equal to $\prod_{1 \leq i \leq \bar l} \prod_{1 \leq j \leq \bar
  k} (b_i - a_j)$, which is indeed non-zero.

We have now proved the linear independence of the vectors
$x v_+$ as $x$ runs over all ordered supermonomials
in $\{f^{(r)}\:|\:s_{2,1} < r \leq s_{2,1}+\bar k\}$.
It remains to show that these vectors also span $\overline{L}(A)$.
For this, it is enough to show that
$\dim \overline{L}(A) \leq 2^{\bar k}$. This will be established in the
next section by means of an explicit construction
of a module of dimension $2^{\bar k}$ containing $\overline{L}(A)$ as
a subquotient.
\end{proof}

\section{Tensor products}

In this section we define some more general comultiplications
between the algebras $W_\pi$, allowing certain tensor products to be
defined. We apply this to construct so-called {\em standard modules}
$\overline{V}(A)$ for each $A \in \Tab_\pi$. Then we complete the proof of Theorem~\ref{moved}
by showing that every irreducible $W_\pi$-module is
isomorphic to one of the modules
$\overline{V}(A)$ for suitable $A$.

Recall that the pyramid $\pi$ has $l$ boxes on its second row.
Suppose we are given $l_1,\dots,l_d \geq 0$ such that
$l_1+\cdots+l_d = l$.
For each $c=1,\dots,d$, let $\pi_c$ be the pyramid
consisting of columns $l_1+\cdots+l_{c-1}+1,\dots,l_1+\cdots+l_c$
of $\pi$.
Thus $\pi$ is the ``concatenation'' of the pyramids
$\pi_1,\dots,\pi_d$.
Let $W_{\pi_c}$ be the principal $W$-algebra
defined from $\pi_c$.
Let $\sigma_1,\dots,\sigma_d$ be the unique shift matrices
such that each $\sigma_c$ is compatible with $\pi_c$,
and $\sigma_c$ is lower (resp.\ upper) triangular if
$s_{2,1} \geq l_1+\cdots+l_c$ (resp.\ $s_{1,2} \geq l_c+\cdots+l_d$).
We denote the Miura transform for $W_{\pi_c}$
by $\mu_c:W_{\pi_c} \hookrightarrow U_{\sigma_c}^{l_c}$.

\begin{Lemma}\label{comultmap}
With the above notation,
there is a unique injective algebra homomorphism
\begin{equation}\label{gencom}
\Delta_{l_1,\dots,l_d}:W_\pi \hookrightarrow W_{\pi_1}
\otimes\cdots\otimes W_{\pi_d}
\end{equation}
such that
$(\mu_1 \otimes \cdots \otimes \mu_d) \circ \Delta_{l_1,\dots,l_d} =
\mu$.
\end{Lemma}

\begin{proof}
Let us add the suffix $c$ to all notation arising from the definition of
$W_{\pi_c}$, so that
$W_{\pi_c}$ is a subalgebra of $U(\p_c)$, we have that
$\g_c = \m_c\oplus\h_c\oplus\p_c^\perp$, and so on.
We identify $\g_1 \oplus \cdots \oplus \g_d$ with a subalgebra
$\g'$ of $\g$
so that $e_{i,j} \in \g_c$
is identified with $e_{i',j'} \in \g$
where $i'$ and $j'$ are the labels of
the boxes of $\pi$ corresponding to the
$i$th and $j$th boxes of $\pi_c$, respectively.
Similarly we identify $\m_1\oplus\cdots\oplus \m_d$
with $\m' \subseteq \m$, $\p_1\oplus\cdots\oplus \p_d$ with
$\p'\subseteq \p$, and $\h_1\oplus\cdots\oplus \h_d$ with
$\h' = \h$.
Also let $\tilde\rho' \coloneqq \tilde\rho_1+\cdots+\tilde \rho_d$,
a character of $\p'$.
In this way $W_{\pi_1}\otimes\cdots\otimes W_{\pi_d}$
is identified with $W_{\pi}' \coloneqq
\{u\in U(\p')\:|\: u\m_\chi' \subseteq \m_\chi'
U(\g')\}$,
where $\m_\chi' = \{x-\chi(x)\:|\:x \in \m'\}$.

Let $\q$ be the unique parabolic subalgebra of $\g$
with Levi factor $\g'$ such that
$\p \subseteq \q$.
Let $\psi:U(\q) \twoheadrightarrow U(\g')$
be the homomorphism induced by the natural projection of
$\q\twoheadrightarrow
\g'$.
The following diagram commutes:
$$
\begin{CD}
U(\p)&@>S_{-\tilde\rho'}\circ\psi\circ S_{\tilde\rho}>>& U(\p')\\
@V\operatorname{pr}\circ S_{\tilde\rho} VV&&@VV\operatorname{pr}'\circ
S_{\tilde\rho'} V\\\
U(\h)&@=&U(\h')
\end{CD}
$$
We claim that
$S_{-\tilde\rho'} \circ \psi \circ S_{\tilde\rho}$
maps $W_\pi$ into $W_\pi'$. The claim implies the lemma,
for then it makes sense to {\em define}
$\Delta_{l_1,\dots,l_d}$ to be the restriction of this map to $W_\pi$,
and we are done by the commutativity of the above diagram and injectivity of
the Miura transform.

To prove the claim,
observe that $\tilde\rho-\tilde\rho'$ extends to a character of
$\q$, hence there is a corresponding shift automorphism
$S_{\tilde\rho-\tilde\rho'}:U(\q) \rightarrow U(\q)$
which preserves $W_\pi'$.
Moreover $S_{-\tilde\rho'}\circ\psi\circ S_{\tilde\rho} =
S_{\tilde\rho-\tilde\rho'} \circ
\psi$.
Therefore it enough to check just that
$\psi(W_\pi) \subseteq W_\pi'$.
To see this, take $u \in W_\pi$, so that
$u \m_\chi \subseteq \m_\chi U(\g)$.
This implies that $u \m_\chi' \subseteq \m_\chi U(\g) \cap U(\q)$,
hence applying $\psi$
we get that $\psi(u) \m_\chi' \subseteq \m_\chi' U(\g')$.
This shows that $\psi(u) \in W_\pi'$ as required.
\end{proof}

\begin{Remark}
Special cases of the maps (\ref{gencom}) with $d=2$
are related to the comultiplications $\Delta, \Delta_+$ and $\Delta_-$
from (\ref{gin})--(\ref{dm}).
Indeed, if $l = l_1+l_2$
for $l_1 \geq s_{2,1}$ and $l_2 \geq s_{1,2}$,
the shift matrices $\sigma_1$ and $\sigma_2$ above are equal to
$\sigma^\lo$ and $\sigma^\up$, respectively.
Both squares in
the following diagram commute:
$$
\begin{CD}
Y_\sigma &@>\Delta>> &Y_{\sigma_1} \otimes Y_{\sigma_2}\\
@V\ev_\sigma^l VV&&@VV\ev_{\sigma_1}^{l_1} \otimes \ev_{\sigma_2}^{l_2}V\\
U_\sigma^l&@=&U_{\sigma_1}^{l_1}\otimes U_{\sigma_2}^{l_2}\\
@A\mu AA&&@AA\mu_1 \otimes \mu_2 A\\
W_\pi &@>\Delta_{l_1,l_2}>>&W_{\pi_1} \otimes W_{\pi_2}
\end{CD}
$$
Indeed, the top square commutes by the definition of the evaluation
homomorphisms from (\ref{Evalhom}),
while the bottom square commutes
by Lemma~\ref{comultmap}.
Hence, under our isomorphism between principal $W$-algebras and truncated
shifted Yangians,
$\Delta_{l_1,l_2}:W_\pi \rightarrow W_{\pi_1} \otimes W_{\pi_2}$
corresponds exactly to the map
$Y^l_\sigma \rightarrow Y^{l_1}_{\sigma_1} \otimes Y^{l_2}_{\sigma_2}$
induced by the comultiplication $\Delta:Y_\sigma \rightarrow
Y_{\sigma_1} \otimes Y_{\sigma_2}$.

Instead,
if $l_1=l-1$, $l_2=1$ and
the rightmost column of $\pi$ consists of a single box,
the map $\Delta_{l-1,1}:W_\pi \rightarrow W_{\pi_1} \otimes
U(\mathfrak{gl}_1)$
corresponds exactly to the map
$Y^l_\sigma \rightarrow Y^{l-1}_{\sigma_+} \otimes U(\mathfrak{gl}_1)$
induced by $\Delta_+:Y_\sigma \rightarrow Y_{\sigma_+} \otimes
U(\mathfrak{gl}_1)$.
Similarly, if $l_1=1,l_2=l-1$ and
the leftmost column of $\pi$ consists of a single box,
$\Delta_{1,l-1}:W_\pi \rightarrow
U(\mathfrak{gl}_1)\otimes W_{\pi_2}$
corresponds exactly to the map
$Y^l_\sigma \rightarrow U(\mathfrak{gl}_1) \otimes Y^{l-1}_{\sigma_-}$
induced by $\Delta_-:Y_\sigma \rightarrow
U(\mathfrak{gl}_1) \otimes Y_{\sigma_-}$.
\end{Remark}

Using (\ref{gencom}), we can make sense of tensor products:
if we are given $W_{\pi_c}$-modules $V_c$ for each $c=1,\dots,d$
then we obtain a well-defined $W_\pi$-module
\begin{equation}
V_1 \otimes\cdots\otimes V_d \coloneqq \Delta_{l_1,\dots,l_d}^*(V_1
\boxtimes\cdots\boxtimes V_d),
\end{equation}
i.e.\ we take the pull-back of their outer tensor product
(viewed as a module via the usual sign convention).

Now specialize to the situation that
$d=l$ and $l_1=\cdots=l_d = 1$. Then each pyramid
$\pi_c$ is a single column of height one
or two. In the former case $W_{\pi_c} =
U(\mathfrak{gl}_{1})$ and in the latter $W_{\pi_c} = U(\mathfrak{gl}_{1|1})$.
So we have that
$W_{\pi_1} \otimes\cdots\otimes W_{\pi_l} = U_\sigma^l$, and
the map $\Delta_{1,\dots,1}$ coincides with the Miura transform
$\mu$.

Given $A \in \Tab_\pi$, let $A_c \in \Tab_{\pi_c}$ be its $c$th column
and $\overline{L}(A_c)$ be the corresponding irreducible
$W_{\pi_c}$-module.
Let us decode this notation a little.
If  $W_{\pi_c} = U(\mathfrak{gl}_1)$ then
$A_c$ has just a single entry $b$ and
$\overline{L}(A_c)$ is the one-dimensional module
with an even basis vector $v_+$ such that
$e_{1,1} v_+ = (-1)^{|2|} b v_+$.
If $W_{\pi_c} = U(\mathfrak{gl}_{1|1})$ then
$A_c$ has two entries, $a$ in the first row and $b$ in the second
row, and $\overline{L}(A_c)$ is one- or two-dimensional according to
whether $a=b$ or not;
in both cases $\overline{L}(A_c)$ is generated by an even vector $v_+$
such that $e_{1,1} v_+ = (-1)^{|1|}a v_+, e_{2,2} v_+ = (-1)^{|2|}bv_+$
and $e_{1,2} v_+ = 0$.
Let
\begin{equation}
\overline{V}(A) \coloneqq \overline{L}(A_1) \otimes\cdots\otimes
\overline{L}(A_l).
\end{equation}
Note that $\dim \overline{V}(A) = 2^{k-h}$
where $h$ is the number of $c=1,\dots,l$ such that
$A_c$ has two equal entries.

\begin{Lemma}\label{ishw}
For any $A \in \Tab_\pi$,
there is a non-zero homomorphism
$$
\overline{M}(A) \rightarrow \overline{V}(A)
$$
sending
the cyclic vector $1 \otimes 1_A \in \overline{M}(A)$
to $v_+ \otimes \cdots \otimes v_+ \in \overline{V}(A)$.
In particular $\overline{V}(A)$ contains a subquotient isomorphic to
$\overline{L}(A)$.
\end{Lemma}

\begin{proof}
Suppose that
$A = \substack{a_1 \cdots a_{k} \\ b_1 \cdots b_{l}}$.
By the definition of $\overline{M}(A)$ as an induced module, it
suffices
to show that $v \coloneqq v_+ \otimes \cdots \otimes v_+ \in \overline{V}(A)$
is annihilated by all $e^{(r)}$ for $r > s_{1,2}$ and that
$d_1^{(r)} v = e_r(a_1,\dots,a_k) v$
and $d_2^{(r)} v = e_r(b_1,\dots,b_l) v$
for all $r > 0$.
For this we calculate from the explicit formulae for
the invariants $d_1^{(r)}, d_2^{(r)}$ and $e^{(r)}$
given by (\ref{EQ:Tpidef}) and (\ref{maini})--(\ref{mainj}), remembering that
their action on $v$ is defined via the Miura transform $\mu = \Delta_{1,\dots,1}$.
It is convenient in this proof
to set
$$
\bar e_{i,j}^{[c]} \coloneqq
\left\{
\begin{array}{ll}
(-1)^{|i|}1^{\otimes (c-1)} \otimes e_{i,j}
\otimes
1^{\otimes (l-c)}&\text{if $q_c=2$,}\\
(-1)^{|2|} 1^{\otimes (c-1)}\otimes e_{1,1} \otimes
1^{\otimes (l-c)}&\text{if $q_c = 1$ and $i=j=2$,}\\
0&\text{otherwise,}
\end{array}
\right.
$$
for any $1 \leq i,j \leq 2$ and $1 \leq c \leq l$, where $q_c$ is the
number of boxes in the $c$th
column of $\pi$.
First we have that
$$
d_1^{(r)} v = \sum_{1 \leq c_1,\dots,c_r \leq l}
\sum_{1 \leq h_1,\dots,h_{r-1} \leq 2}
\bar e_{1,h_1}^{[c_1]}
\bar e_{h_1,h_2}^{[c_2]}
\cdots \bar e_{h_{r-1},1}^{[c_r]} v
$$
summing only over terms with $c_1 < \cdots < c_r$.
The elements on the right commute (up to sign) because
the $c_i$ are all distinct,
so any $\bar e_{1,2}^{[c_i]}$ produces
zero as $e_{1,2} v_+ = 0$.
Thus the summation reduces just to
$$
\sum_{1 \leq c_1 < \cdots < c_r \leq l}
\bar e_{1,1}^{[c_1]} \cdots \bar e_{1,1}^{[c_r]} v = e_r(a_1,\dots,a_k) v
$$
as required.
Next we have that
$$
d_2^{(r)} v = \sum_{1 \leq c_1, \dots,
c_r \leq l}
\sum_{1 \leq h_1,\dots,h_{r-1} \leq 2}
(-1)^
{\#\left\{i=1,\dots,r-1\:|\:\row(h_i) =1
\right\}}
\bar e_{2,h_1}^{[c_1]}
\bar e_{h_1,h_2}^{[c_2]} \cdots \bar e_{h_{r-1},2}^{[c_r]} v
$$
summing only over terms with
$c_i \geq c_{i+1}$ if $\row(h_i) = 1$,
$c_i < c_{i+1}$ if $\row(h_i) = 2$.
Here, if any monomial $\bar e_{1,2}^{[c_i]}$ appears, the rightmost such
can be commuted to the end, when it acts as zero.
Thus the summation reduces just to the terms with $h_1 = \cdots =
h_{r-1}=2$
and again we get the required elementary symmetric function
$e_r(b_1,\dots,b_l)$.
Finally we have that
$$
e^{(r)} v = \sum_{1 \leq c_1, \dots,
c_r \leq l}
\sum_{1 \leq h_1,\dots,h_{r-1} \leq 2}
(-1)^{\#\left\{i=1,\dots,r-1\:|\:
\row(h_i) =1
\right\}}
\bar e_{1,h_1}^{[c_1]}
\bar e_{h_1,h_2}^{[c_2]} \cdots \bar e_{h_{r-1},2}^{[c_r]} v
$$
summing only over terms with
$c_i \geq c_{i+1}$ if $\row(h_i) = 1$,
$c_i < c_{i+1}$ if $\row(h_i) = 2$.
As before this is zero because the rightmost $\bar e_{1,2}^{[c_i]}$ can
be commuted to the end.
\end{proof}

\begin{Theorem}\label{fint}
Take any $A
=\substack{a_1 \cdots a_{k} \\ b_1 \cdots b_{l}}
\in \Tab_\pi$
and let $h \geq 0$ be maximal such that there exist distinct $1 \leq
i_1,\dots,i_h \leq k$ and distinct
$1 \leq j_1,\dots,j_h \leq l$ with $a_{i_1} = b_{j_1},...,a_{i_h} =
b_{j_h}$.
Choose
$B \sim A$ so that $B$ has $h$ columns of height two containing equal
entries.
Then
\begin{equation}
\overline{L}(A) \cong \overline{V}(B).
\end{equation}
In particular $\dim \overline{L}(A) = 2^{k-h}$.
\end{Theorem}

\begin{proof}
By Lemma~\ref{ishw}, $\overline{V}(B)$ has a subquotient isomorphic to
$\overline{L}(B) \cong \overline{L}(A)$,
which implies that $\dim \overline{L}(A) \leq \dim \overline{V}(B) = 2^{k-h}$.
Also by the linear independence established in the partial
proof of Theorem~\ref{moved} given in the previous section
we know that $\dim \overline{L}(A) \geq 2^{k-h}$.
\end{proof}

In particular this establishes the fact about dimension needed to complete
the proof of Theorem~\ref{moved}
in the previous section.

\end{document}